\documentclass{amsart}
\usepackage[top=0.75in,left=0.75in,right=1in]{geometry}
\usepackage{amsmath,amsthm,amssymb, float}
\usepackage{tabularray}
\usepackage{colortbl}
\usepackage{enumerate}
\usepackage{comment, hyperref}
\usepackage{todonotes}
\usepackage{tabularray}
\usepackage{tikz,pgffor}
\usepackage{arydshln, graphicx}
\usepackage[table,x11names]{xcolor}
\usetikzlibrary{positioning,calc, arrows.meta}
\usetikzlibrary{decorations.pathreplacing}
\usetikzlibrary{decorations.pathreplacing,calligraphy}

\newtheorem{thm}{Theorem}[section]
\newtheorem{lem}[thm]{Lemma}
\newtheorem{prop}[thm]{Proposition}
\newtheorem{cor}[thm]{Corollary}
\newtheorem*{cor*}{Corollary}
\newtheorem{conj}[thm]{Conjecture}
\newtheorem*{claim}{Claim}
\newtheorem{question}{Question}
\newtheorem*{question*}{Question}

\theoremstyle{definition}
\newtheorem{defi}[thm]{Definition}
\newtheorem*{rem}{Remark}

\newtheorem{example}[thm]{Example}

\DeclareMathOperator{\mex}{mex}
\DeclareMathOperator{\dom}{dom}

\title{Two Dimensional Subtraction - Transfer Games}

\author[A. Danai]{Alon Danai}
\address{Department of Mathematics, Rutgers University,
110 Frelinghuysen Road,
Piscataway, NJ,
08854,
USA}
\email{}
\author[P. Ellis]{Paul Ellis}
\address{Department of Mathematics, Rutgers University,
110 Frelinghuysen Road,
Piscataway, NJ,
08854,
USA}
\email{paulellis@paulellis.org}
\author[T. A. Thanatipanonda]{Thotsaporn Aek Thanatipanonda}
\address{Science Division, Mahidol University International College,
999 Phutthamonthon Sai 4 Rd, Salaya, Phutthamonthon District,
Nakhon Pathom,
73170,Thailand}
\email{thotsaporn@gmail.com}

\begin{document}
\maketitle
\begin{abstract}

We generalize the results and conjectures in \cite{Lengyel}, showing that the \textsc{nim}-values of a large class of two-dimensional subtraction-transfer games are periodic.  These are impartial, normal-play games with two piles of tokens, where players alternate either taking some tokens from a pile or transferring tokens from one pile to the other.
In many cases, we calculate the exact period.  We also develop several new notions of periodicitiy.
  
\end{abstract}

\section{Introduction}

In \cite{Lengyel}, Tam\'{a}s Lengyel investigated variations of a normal-play, impartial,  two-pile subtraction transfer game.  All of his examples involved two piles of tokens, where a player was allowed to either subtract a particular number from the first pile, subtract another particular number from the second pile, or move a single token from the first pile to the second pile.  He then proved that several games of this type are \emph{periodic}.  That is, there is a \emph{period} $(p,q)\in\mathbb{N}^2$ so that, for any $(x,y)\in\mathbb{N}^2$, $\mathcal{SG}(x,y)=\mathcal{SG}(x+p,y)=\mathcal{SG}(x,y+q)$, where $x$ and $y$ denote the number of tokens in the first and second pile, respectively.  In some sense, this paper merely presents a generalization of his results.

However, we also challenge a method in \cite{Larsson}.  In this paper, Larsson, Saha, and Yokoo investigated one-, two-, and three-move vector subtraction games.  In doing so, they focus only on the outcome class of game positions, ignoring \textsc{nim}-values.  
On one hand, studying only the outcome classes gives enough information for their purposes, but it is also a common point of view to consider, for a given impartial game, the calculation of outcome classes to be easier than a full calculation of \textsc{nim}-values.
It turns out that this point of view would not have allowed us to generalize Lengyel's results, even though it `should' be easier to consider only outcome classes.  By considering outcome classes, one essentially partitions all \textsc{nim}-values as:
\[0 \mid 1 2 3 4 \ldots,\]
whereas in the study of what we are calling \emph{Lengyel transfer games}, it is more fruitful to mentally partition  \textsc{nim}-values according to $a\sim a\oplus 1$, that is, as:
\[0 1 \mid  2 3 \mid  4 5 \mid 6 7 \mid 89 \ldots\]
Furthermore, we end up discovering some new notions of periodicity.






\begin{defi}
    A \emph{vector game} is an impartial, normal-play, two player game played on $\mathbb{N}^n$ for some $n$.  It is determined by a set of vectors $S=\{u_i\}_{i\in I}\subseteq \mathbb{Z}^n$.  A legal move consists of adding an element of $S$ to the current game position.  If each $u_i$ is an element of $\left(\mathbb{Z}^{\leq 0}\right)^n$, we call it a \emph{vector subtraction game}.
\end{defi}

To ensure the game is not loopy, we further require that adding an element of $S$ reduces some monovariant.  In the case of vector subtraction games,  this could be the sum of the components.  In all of our examples, it will be the lexicographic order on  $\mathbb{N}^2$, that is, $(x,y)\leq (x',y')$ if either $x<x'$ or both $x=x'$ and $y\leq y'$.


\begin{defi}
 Define the \emph{Lengyel transfer game} $L(b;x_1,y_1;x_2,y_2)$ as the vector game defined by \newline
 $\{(0,-b), (-x_1,y_1), (-x_2,y_2)\}$, where $b,x_1,x_2\in\mathbb{N}^+$ and $y_1,y_2\in\mathbb{N}$.  
\end{defi}

Lengyel proved the following.

\begin{thm}[\cite{Lengyel}]\;

    \begin{enumerate}[(a)]
        \item $L(1;1,0;1, 1)$ has period $(2,2)$.
        \item If $b$ is even, $L(b;b,0; 1, 1)$ has period $(2b,2b)$.  In fact, $(x,y)$ is a  $\mathcal{P}$-position precisely when $\lfloor\frac{x+y}{b}\rfloor+x$ is even, which solves the olympiad problem which motivated \cite{Lengyel}.
        \item If $x_1\geq 2$, the period of $L(1;x_1,0; 1, 1)$ is $(2(x_1+1),2)$.
    \end{enumerate}
    \label{thm:lengyel}
\end{thm}

He then conjectured the following.

\begin{conj}[\cite{Lengyel}]\label{lengyel's conjecture}\;

    \begin{enumerate}[(a)]
        \item If $b$ is odd and at least $3$
        , $L(b;b,0;1, 1)$ has  period $(2b(b+1),2b)$.
        \item For $b\geq 2$, $L(b;1,0;1, 1)$ has period $(4b,2b)$
    \end{enumerate}
\end{conj}

We will generalize Theorem \ref{thm:lengyel} and Conjecture \ref{lengyel's conjecture} together as follows.  

\begin{thm}[First Main Theorem]\;
\begin{enumerate}[(a)]
\item $L(1;1,0;1, 1)$ has period $(2,2)$.

\item If $b$ is even, and if $x_1$ is an odd multiple of $b$, then  $L(b;x_1,0;1, 1)$ has period  $(2b,2b)$.
\item In all other cases, $L(b;x_1,0;1, 1)$ has period $(2b(x_1+1), 2b)$.
\end{enumerate}  
\label{Main Theorem 1}
\end{thm}

As an example to aid visualization, Figures \ref{fig:L(1;1,0;1,1)} and \ref{fig:L(2;3,0;1,1)} show the array of \textsc{nim}-values for $L(1;1,0;1,1)$ and $L(2;3,0;1,1)$.  We can think of the game as being played on an infinite $\mathbb{N}\times\mathbb{N}$ board, where the origin is in the top left corner.  In the latter example, a legal move is to jump up $2$ units, left $3$ units, or down-left $1$ unit ($\sqrt{2}$ units?).  To aid clarity, we will freely conflate this table of \textsc{nim}-values with the game board itself.  Note that the initial column and row are each numbered $0$, as they represent empty piles of tokens.

\begin{figure}[h]
    \centering
    $\begin{array}{cc|cc}
         0 & 2 & 0 & 2  \\
         1 & 3 & 1 & 3   \\ \hline
         0 & 2 & 0 & 2   \\
         1 & 3 & 1 & 3 
    \end{array}$
    \caption{Four periods of the array of \textsc{nim}-values for $L(1;1,0;1,1)$}
    \label{fig:L(1;1,0;1,1)}
\end{figure}

\begin{figure}[h]
    \centering
    $\begin{tblr}{cccc|[dashed]cccc|[dashed]cccc|[dashed]cccc}
         0 & 1 & 1 & 2 & 0 & 0 & 1 & 1 & 1 & 0 & 0 & 2 & 1 & 1 & 0 & 0  \\
         0 & 0 & 1 & 1 & 1 & 0 & 0 & 2 & 1 & 1 & 0 & 0 & 0 & 1 & 1 & 2  \\ \hline[dashed]
         1 & 0 & 0 & 3 & 1 & 1 & 0 & 0 & 0 & 1 & 1 & 3 & 0 & 0 & 1 & 1  \\
         1 & 1 & 0 & 0 & 0 & 1 & 1 & 3 & 0 & 0 & 1 & 1 & 1 & 0 & 0 & 3
    \end{tblr}$
    \caption{One period of the array of \textsc{nim}-values for $L(2;3,0;1,1)$}
    \label{fig:L(2;3,0;1,1)}
\end{figure}

In Section \ref{Section-nim periodicity}, we carefully develop the notion \textsc{nim}-periodicity.  Our first result in that section will imply the following.
\begin{lem}\label{lem-2b}
The second component of the period of the Lengyel transfer game $L(b;x_1,y_1;x_2,y_2)$ is always $2b$. 
\end{lem}
While Lemma \ref{lem - eventual horizontal periodicity} says that all Lengyel transfer games are \emph{eventually} periodic in the first component. 

In Section \ref{section-diagonal periodicity}, we extend a key lemma of \cite{Larsson}, and use it to develop the notion of diagonal periodicity (Definition \ref{def: diagonal periodicity}).  In Section \ref{Section - Proof of Main Theorem 1}, we prove Theorem \ref{Main Theorem 1}. Section \ref{section - long preperiods} shows that we may have arbitrarily long preperiods.  In Section \ref{Section - simplifying assumptions}, we organize some results which simplify the remaining sections.  In Section \ref{section - eventual diagonal periodicity}, we prove that in most cases, eventual diagonal periodicity is obtained, leading to the following.
\begin{thm}[Second Main Theorem]\label{conjecture - big one}
    The first component of the period of the Lengyel transfer game $L(b;x_1,y_1;x_2,y_2)$ is a factor of
    \[g(b,x_1,y_1,x_2,y_2)=\frac{2b(x_1+x_2)}{\gcd\{2b,y_1+y_2\}}.\]
\end{thm}
Note that $g(b,x_1,0,1,1)=2b(x_1+1)$, matching Theorem \ref{Main Theorem 1}(c).

In Section \ref{section - $L(b;x_1,0;x_2,y_2)$} we treat the periods of all games of the form $L(b;x_1,0,x_2,y_2)$, though we must leave the final result as conjecture.  
In Section \ref{section - lengyel's other conjecture} we prove the remaining conjecture from \cite{Lengyel}, which is about a class of games with multiple transfer options.   In Section \ref{section - open questions}, we explore some examples which lead to open questions.






\begin{rem}
    We extensively used computer models to find our results.  In doing so, if the period of $L(b;x_1,y_1;x_2,y_2)$ appears to be $(p,q)$, then to verify it as such, we only need to compute a block of size $(2m)\times(2n)$, where $m=\max\{p,x_1,x_2\}$ and $n=\max\{q,b,y_1,y_2\}$.  Hence, Figure \ref{fig:L(1;1,0;1,1)} constitutes a proof of Theorem \ref{Main Theorem 1}(a).
\end{rem}

\section{\textsc{Nim}-Periodicity}\label{Section-nim periodicity}

Lemma \ref{lem-2b} is a direct consequence of a more subtle phenomenon.  

\begin{lem}[\textsc{Nim}-Periodicity]
For any  $L(b;x_1,y_1;x_2,y_2)$, $\mathcal{SG}(x,y+b)=\mathcal{SG}(x,y)\oplus 1$.
\end{lem}  
For example, note that $\mathcal{SG}(x,y+2)=\mathcal{SG}(x,y)\oplus1$ for all $(x,y)$ in Figure \ref{fig:L(2;3,0;1,1)}, and $\mathcal{SG}(x,y+1)=\mathcal{SG}(x,y)\oplus1$ for all $(x,y)$ in Figure \ref{fig:L(1;1,0;1,1)}.  

Note that the following proof can be generalized to any vector game of the form  $\{(0,-b), (-x_i,y_i)_{i\in I}\}$, where each $x_i>0$, each $y_i\geq 0$, and $I$ is finite.
\begin{proof}
Consider some $L(b;x_1,y_1;x_2,y_2)$, and suppose $\mathcal{SG}(x,y)=a$. We proceed by induction on the lexicographic order of $(x,y)$.  Let $S$ be the subset of $\{(x-x_1,y+y_1), (x-x_2,y+y_2)\}$ whose components are all nonnegative.  Define $\mathcal {SG}(S)$ as the corresponding set of values.  Define $S+b$ as the corresponding subset of $\{(x-x_1,y+y_1+b), (x-x_2,y+y_2+b)\}$, and similarly for $\mathcal {SG}(S+b)$.  Note that the options of $(x,y+b)$ are $S+b$ and $(x,y)$.

If $a=0$, then $\mathcal {SG}(S)$ does not contain $0$, so by induction $\mathcal {SG}(S+b)$ does not contain $1$.  Hence $\mathcal{SG}(x,y+b)=1$.

If $a=1$, then $\mathcal {SG}(S)$ does not contain $1$, so by induction $\mathcal {SG}(S+b)$ does not contain $0$.  
Hence $\mathcal{SG}(x,y+b)=0$.

If $a\in\{2,3\}$, then by induction, $\mathcal {SG}(x,y-b)$, if it exists, is in $\{2,3\}$, so $\mathcal {SG}(S)=\{0,1\}$. Then by induction, $\mathcal {SG}(S+b)=\{0,1\}$.  Hence $\mathcal{SG}(x,y+b)=\mex\{0, 1, a\}=a\oplus 1$.
\end{proof}

Lemma \ref{lem-2b} follows immediately. So from now on we focus on the first component of the period $L(b; x_1; y_1; x_2, y_2)$.

\begin{defi}
    An array $\{\mathcal{SG}(x)\}_{x\in\mathbb{N}^n}$ is \textit{periodic in $u$} for some $u\in\mathbb{N}^n$ if $\mathcal{SG}(x)=\mathcal{SG}(x+u)$ for all $x\in\mathbb{N}^n$ 
\end{defi}

In other words, any Lengyel transfer game $L(b;x_1,y_1;x_2,y_2)$ is periodic in $(0,2b)$.  If it is also periodic in $(p,0)$, we either say that such a game has \emph{horizontal periodicity} (with period $p$), or simply say that it is periodic with period $p$.

The next three results are almost immediate consequences of \textsc{nim}-periodicity.

\begin{cor}\label{cor: y can be reduced mod 2b}
    If $y_1\equiv y'_1\bmod 2b$ and $y_2\equiv y'_2\bmod 2b$, then $L(b;x_1,y_1;x_2,y_2)$ and $L(b;x_1,y'_1;x_2,y'_2)$ have the same \textsc{nim}-values.
\end{cor}

\begin{proof}
    We proceed by induction on the lexicographic order of $(x,y)$. Assume that both games agree about the \textsc{nim}-values of all points less than $(x,y)$.  By \textsc{nim}-periodicity, $\mathcal{SG}(x-x_i,y+y_i)=\mathcal{SG}(x-x_i,y+y'_i)$, $1\leq i\leq 2$. Thus both games agree about the value of $\mathcal{SG}(x,y)$.
\end{proof}

And so we restrict our analysis to the case $0\leq y_1, y_2<2b.$

\begin{cor}\label{corollary: 2s vs 3s}
    In $L(b;x_1,y_1;x_2,y_2)$, the rows numbered $0,\ldots,b-1 \bmod 2b$ cannot contain any $3$s.  The other rows cannot contain any $2$s
\end{cor}

\begin{proof}
    Each of the positions in the first $b$ rows have at most $2$ options.  So these rows only contain the \textsc{nim}-values $0$, $1$, and $2$.  The rest follows by induction and \textsc{nim}-periodicity.
\end{proof}

\textsc{nim}-periodicity also allows us to show that all Lengyel transfer games are eventually periodic.

\begin{lem}\label{lem - eventual horizontal periodicity}
    Let $L(b;x_1,y_1;x_2,y_2)$  be a Lengyel transfer game.  Then it is eventually horizontally periodic.
\end{lem}

\begin{proof}
    Our first step is to show that $L(b;x_1,y_1;x_2,y_2)$ has eventual row periodicity, as defined in \cite{Larsson}.

    
    Let $M=\max\{x_1,x_2\}$. Then \textsc{nim}-periodicity implies that each block of $M$ columns is a function of the first $2b$ rows of the previous $M$ columns.  So by the pigeonhole principle, these must eventually repeat.  
\end{proof}


We may also restrict our analysis to the cases where the horizontal/vertical components are relatively prime.  

\begin{lem}
    For any vector game defined by $(x_i,y_i)_{i\in I}$, and any positive integers $m,n$, the array of nim values of the game  $(m\cdot x_i,n\cdot y_i)_{i\in I}$ is given by that of $(x_i,y_i)_{i\in I}$, but with each entry replaced by an $m\times n$ block copy of itself.  In particular, if the period of $(x_i,y_i)_{i\in I}$ is $(p,q)$ with $p,q>1$, then the period of $(m\cdot x_i,n\cdot y_i)_{i\in I}$ is $(mp,nq)$.  In the case $p=1$ or $q=1$, that part of the period remains the same.
\end{lem}

Note that this lemma is easily generalizable to more than $2$ dimensions, but we state it this way for clarity.

\begin{proof}
    Essentially \cite[Proposition 2.2]{Lengyel}.
\end{proof}

\begin{cor}\label{lemma:dilation}
    If the period of $L(b;x_1,y_1;x_2,y_2)$ is $(p,q)$, with $p,q>1$, then the period of $L(nb;mx_1,ny_1;mx_2,ny_2)$ is $(mp,nq)$.  If either of $p$ or $q$ is $1$, then that component of the period remains $1$.
\end{cor}
Hence we may restrict our analysis to the case when $\gcd\{x_1,x_2\}=\gcd\{b,y_1,y_2\}=1$.
Note also that \[g(nb,mx_1,ny_1,mx_2,ny_2) =\frac{2nb(mx_1+mx_2)}{\gcd\{2nb,ny_1+ny_2\}}=\frac{(nm)2b(x_1+x_2)}{(n)\gcd\{2b,y_1+y_2\}}=m\cdot g(b,x_1,y_1,x_2,y_2),\] 
so if Theorem \ref{conjecture - big one} holds for $L(b;x_1,y_1;x_2,y_2)$, then it holds for $L(nb;mx_1,ny_1;mx_2,ny_2)$.

We are now ready to prove a degenerate version of Theorem \ref{Main Theorem 1}, which will also serve as a step in its proof. Let $L(b;x_1,y_1)$ denote the vector game defined by $\{(0,-b),(-x_1,y_1)\}$.  Figure \ref{fig:two-move case} shows the initial periods for the games $L(2;1,1)$ and $L(3;1,1)$. 

\begin{figure}[h]
    \centering
    $\begin{tblr}{cccc|cccc}
         0 & 1 & 1 & 0 & 0 & 1 & 1 & 0 \\
         0 & 0 & 1 & 1 & 0 & 0 & 1 & 1  \\ \hline[dashed]
         1 & 0 & 0 & 1 & 1 & 0 & 0 & 1  \\
         1 & 1 & 0 & 0 & 1 & 1 & 0 & 0
    \end{tblr}$ \qquad $\begin{tblr}{ccc|ccc}
         0 & 1 & 0 & 0 & 1 & 0 \\
         0 & 1 & 1 & 0 & 1 & 1 \\
         0 & 0 & 1 & 0 & 0 & 1 \\ \hline[dashed]
         1 & 0 & 1 & 1 & 0 & 1 \\
         1 & 0 & 0 & 1 & 0 & 0 \\
         1 & 1 & 0 & 1 & 1 & 0 
    \end{tblr}$
    \caption{Two initial periods of the arrays of \textsc{nim}-values for $L(2;1,1)$ and $L(3;1,1)$}
    \label{fig:two-move case}
\end{figure}

\begin{thm}[Zeroeth Main Theorem]\label{Theorem: 2 vector case}
    If $y_1$ is an odd multiple of $b$, then 
    the game $L(b;x_1,y_1)$ has period $(1,2b)$.  Otherwise, it has period $\left(\frac{2bx_1}{\gcd\{2b,y_1+b\}},2b\right)$. 
\end{thm}  

Note that $\frac{2bx_1}{\gcd\{2b,y_1+b\}}=g(b,x_1,y_1,0,-b)$, so this is also a kind of degenerate case of Theorem \ref{conjecture - big one}. 

\begin{proof}
In light of Corollary \ref{lemma:dilation}, we may restrict our attention to the case $x_1=1$.

The first $2b$ entries in the  first column are $b$ $0$s followed by $b$ $1$s.  After this,  each column is a repeat of the previous column, but shifted up by $y_1$ positions and $\oplus 1$-ed.  In other words, it is as if each column is an upward shift of the previous column by $y_1+b$ units.  Hence it takes $\frac{2b}{\gcd\{2b,y_1+b\}}$ of these shifts for a column to repeat.

Note that $\gcd\{2b,y_1+b\}=2b$ precisely when $y_1$ is an odd multiple of $b$.  So in this case $\frac{2b}{\gcd\{2b,y_1+b\}}=1$.
\end{proof}

\begin{cor}\label{cor: period of L(b,1,1)}
    The period of $L(b;1,1)$ is $(2b,2b)$ if $b$ is even, and $(b,2b)$ if $b$ is odd.
\end{cor}

        

\subsection{Generalized \textsc{Nim}-Periodicity}

One-dimensional vector subtraction games satisfy the Ferguson pairing property, which states that positions of value $0$ and $1$ are paired by the smallest move.

\begin{prop}\cite{Ferguson}
    Given $n_1<n_2<\ldots<n_k$, then in \textsc{subtract}$(n_1,n_2,\ldots,n_k)$, $\mathcal{SG}(x)=1$ if and only if $\mathcal{SG}(x-n_1)=0$.
\end{prop}

In $L(b;x_1,y_1;x_2,y_2)$, we can hear the vector $(0,-b)$ as the `smallest' move in the lexicographic order.  If we do, then \textsc{Nim}-periodicity rhymes with the Ferguson pairing property.

The idea behind \textsc{nim}-periodicity applies to a larger class of games, so we also define it in a broader context. To see how the general version is applied to Lengyel transfer games, we first redefine vector games in terms of partial functions.  In particular, if a vector game is defined by $S=\{u_i\}_{i\in I}\subseteq \mathbb{Z}^n$, then we can instead think of each $u_i$ as a partial function $f_i$ on $\mathbb{N}^n$, which adds $u_i$ when possible.  Then the options of any position $v$ are precisely $\{f_i(v)\mid i\in I \text{ and }v\in \dom f_i\}$. So in $L(b;x_1,y_1;x_2,y_2)$, the partial functions are $f_1(x,y)=(x-x_1,y+y_1), f_2(x,y)=(x-x_2,y+y_2), f_3(x,y)=(x,y-b)$, with, for example, $\dom f_3=\{(x,y)\mid y\geq b\}$. 

\begin{defi}
    Suppose $R$ is an impartial ruleset where the options of any game position $x\in R$ are given by the partial functions $\{f_i\}_{i\in I}$.  We say $R$ is \emph{\textsc{nim}-periodic in $f_i$} if for all $x\in \dom f_i$, $\mathcal{SG}(f_i(x))=\mathcal{SG}(x)\oplus 1$. In particular, $\mathcal{SG}(f_i(f_i(x)))=\mathcal{SG}(x)$.
\end{defi}

\begin{lem}[Generalized \textsc{Nim}-Periodicity]
    Suppose $R$ is an impartial ruleset where the options of any game position $x\in R$ are given by the partial functions $\{f_i\}_{i\in I}$.  Suppose there is $i\in I$ with the following properties:
    \begin{enumerate}
        \item[(i)] for all $j\in I$, $f_i\circ f_j=f_j\circ f_i$ when defined
        \item[(ii)] for all $x\in R$, if $f_i(x)$ exists, then for all $j\in I, j\neq i$, 
        \[f_j(x)\text{ exists }\iff f_j(f_i(x))\text{ exists}\iff f_i(f_j(x))\text{ exists}\]
    \end{enumerate}
    then $R$ is \textsc{nim}-periodic in $f_i$
\end{lem}

Note that all vector games satisfy the first condition, and that for $L(b;x_1,y_1;x_2,y_2)$, the second condition is satisfied by $f_i(x,y)=(x,y-b)$, showing that, as we have seen, Lengyel transfer games are \textsc{nim}-periodic in $(0,-b)$.

Every part of condition (ii) is necessary.  For example, consider the game $\{(0, -1), (-2, 1)\}$, which is \textsc{nim}-periodic in $(0,-1)$.  Suppose we say that a player cannot do the move $(0, -1)$ if the $y$-value is a multiple of 3. This no longer satisfies 
\[f_j(f_i(x)) \text{ exists} \iff f_i(f_j(x)) \text{ exists},\]
and we lose \textsc{nim}-periodicity. 

\begin{proof}
    Suppose $R$ and $f_i$ are as stated.  Consider $x\in R$ so that $f_i(x)$ exists. The conditions of the lemma imply that, for all $x$, $\{\mathcal{SG}(f_j(f_i(x)))\}_{j\neq i}=\{\mathcal{SG}(f_i(f_j(x)))\}_{j\neq i}$.   We proceed by structural induction on $x\in \dom f_i$.  Fix such an $x$, and let  $a=\mathcal{SG}(f_i(x))$. By induction,
    \begin{enumerate}[(A)]
        \item for all $j$, if $f_i(f_j(x))$ exists, then $\mathcal{SG}(f_i(f_j(x)))=\mathcal{SG}(f_j(x))\oplus 1$; and
        \item $\mathcal{SG}(f_i(f_i(x)))$, if it exists, is $a\oplus 1$.
    \end{enumerate}

    If $a$ is even, $a\oplus 1=a+1$, so by (B), $\{\mathcal{SG}(f_j(f_i(x)))\}_{j\neq i}=\{\mathcal{SG}(f_i(f_j(x)))\}_{j\neq i}$ contains $\{0,1,\ldots, a-1\}$ and not $a$.  Then by (A), $\{\mathcal{SG}(f_j(x))\}_{j\in I}$ contains $\{0\oplus 1,1\oplus 1,\ldots, (a-1)\oplus 1\}=\{0,1,\ldots, a-1\}$ and not $a\oplus 1=a+1$.  But then $\mathcal{SG}(f_i(x))=a$, so  $\mathcal{SG}(x)=a+1=a\oplus 1$.
    
    If $a$ is odd, $a\oplus 1=a-1$.  Whether or not $\mathcal{SG}(f_i(f_i(x)))$ exists, $\{\mathcal{SG}(f_j(f_i(x)))\}_{j\neq i}=\{\mathcal{SG}(f_i(f_j(x)))\}_{j\neq i}$ contains $\{0,1,\ldots, a-2\}$ and not $a$.  Then by (A), $\{\mathcal{SG}(f_j(x))\}_{j\in I}$ contains $\{0\oplus 1,1\oplus 1,\ldots, (a-2)\oplus 1\}=\{0,1,\ldots, a-2\}$ and not $a\oplus 1=a-1$.  Also, $\mathcal{SG}(f_i(x))=a$, so  $\mathcal{SG}(x)=a-1=a\oplus 1$.
\end{proof}

\begin{example}\label{example nim-periodicity}
    The vector game defined by 
    \[\{(-3,0,5),(-2,1,0),(-1,1,1),(0,-3,0),(0,0,-4)\}\]
    is \textsc{nim}-periodic in $(0,-3,0)$ and in $(0,0,-4)$.
\end{example}

It is a direct consequence of the Sprague-Grundy theory that if every position of an impartial game has at most $k$ options, then the the maximum \textsc{nim}-value is at most $k$. In the proof of Theorem \ref{Theorem: 2 vector case}, we see a two-move game whose maximum \textsc{nim}-value is only $1$. In fact, this is not an accident, but a direct consequence of \textsc{nim}-periodicity.  That is, the positions in the first $b$ rows each only have one option, so these values can only be $0$ or $1$.  Then by induction and \textsc{nim}-periodicity, the remaining rows can only have values from $\{0,1\}$ or $\{0\oplus 1,1\oplus 1\}=\{0,1\}$.  More generally, we have the following.

\begin{lem}
    Suppose $R$ is an impartial ruleset where the options of any game position $x\in R$ are given by the partial functions $\{f_i\}_{1\leq i\leq k}$.  Suppose further that $R$ is \textsc{nim}-periodic in $f_i$ for all $1\leq i\leq l$.  Then the maximum \textsc{nim}-value for $R$ is 
    \[\max\{k-l,k-l\oplus1\}\]
\end{lem}

So for example, in Example \ref{example nim-periodicity}, $k=5$ and $l=2$, so the maximum \textsc{nim}-value is $3$.
%

\begin{proof}
    Let $S=\{x\in R\mid x\notin \dom f_i,\text{ for all }1\leq i\leq l\}$.  Then the maximum \textsc{nim}-value of any element of $S$ is $k-l$.  Then by \textsc{nim}-periodicity, the maximum \textsc{nim}-value of any element of $R$ is $\max\{k-l,k-l\oplus1\}$.
\end{proof}

\section{Two-Move Games, Diagonal Periodicity, and $2$-Columns}\label{section-diagonal periodicity}
 In \cite{Larsson}, Larsson, Saha, and Yokoo made an extensive study of two- and three-move vector subtraction games.  In particular, they proved the following.

\begin{lem}\cite{Larsson}\label{Larssons lemma}
    In the two-move vector subtraction game defined by the vectors $\{-u, -v\}$, for any position $x$, the outcome class of $x$ is the same as that of $x+u+v$. 
\end{lem}

Since the maximum \textsc{nim}-value of a two-move impartial game is $2$, the following refinement might not seem significant.  However, our analysis of Lengyel transfer games will rely on it. 
Furthermore, \textsc{nim}-periodicity already shows that it is useful to group our \textsc{nim}-values as $\{0,1\mid 2, 3\}$ rather than as $\{0\mid 1, 2, 3\}$.

\begin{lem}[Two-move periodicity]\label{lemma: refinement of Larsson}
    The two-move vector subtraction game defined by the vectors $\{-u, -v\}$ is periodic in $u+v$. 
\end{lem}

\begin{proof}
    In a two move game, the only possible \textsc{nim}-values are $0$, $1$, and $2$.  Use Figure \ref{Proof of two move lemma} to aid visualization.

    If $\mathcal{SG}(x)=0$, then $\mathcal{SG}(x+v)$ and $\mathcal{SG}(x+u)$ are both positive.  Hence $\mathcal{SG}(x+u+v)=0$.

    If $\mathcal{SG}(x)=1$, then $\mathcal{SG}(x+u), \mathcal{SG}(x+v) \in \{0,2\}$.  We claim that at least one of these values is $0$.  Indeed if $\mathcal{SG}(x+u)= \mathcal{SG}(x+v)=2$, then $\mathcal{SG}(x+u-v)= \mathcal{SG}(x+v-u)=0$.  However, this implies that neither of $\mathcal{SG}(x-u)$ or $\mathcal{SG}(x-v)$, if they exist, are zero, contradicting that $\mathcal{SG}(x)=1$.

    If $\mathcal{SG}(x)=2$, then $\{\mathcal{SG}(x-u), \mathcal{SG}(x-v)\} = \{0,1\}$.  So by the previous two cases, $\{\mathcal{SG}(x+v), \mathcal{SG}(x+u)\} = \{0,1\}$.  Thus $\mathcal{SG}(x+u+v)=2$.   
\end{proof}

\begin{figure}[h]
    \centering
    \begin{tikzpicture}[scale=0.7] 
    \node[black] (A)  at (0,0) {$x$};
    \node[black] (B) at (4,0.4) {$x+u$};
    \node[black] (C) at (0.6,1.1) {$x+v$};
    \node[black] (D) at (4.6,1.5) {$x+u+v$};
    \node[blue] (E) at (-3.4,0.7) {$x+v-u$};
    \node[blue] (F) at (3.4,-0.7) {$x+u-v$};
    \node[red] (G) at (-4,-0.4) {$x-u$};
    \node[red] (H) at (-0.6,-1.1) {$x-v$};
    \draw [->] (A) -- (B);
    \draw [->] (A) -- (C);
    \draw [->] (B) -- (D);
    \draw [->] (C) -- (D);
    \draw [->] (G) -- (E);
    \draw [->] (H) -- (F);
    \draw [->] (E) -- (C);
    \draw [->] (F) -- (B);
    \draw [->] (G) -- (A);
    \draw [->] (H) -- (A);
    \end{tikzpicture}
    \caption{Proof of Lemma \ref{lemma: refinement of Larsson}}
    \label{Proof of two move lemma}
\end{figure}

Note that the $0$ case follows without looking at any options of $x$, the $1$ case requires looking only at the immediate options of $x$, and the $2$ case follows from the other two cases.  In other words, games may exhibit this behavior even if they are only `locally' a $2$-move game.

Due to \textsc{nim}-periodicity, many Lengyel transfer games $L(b;x_1,y_1;x_2,y_2)$ exhibit a similar periodicity as the two-move vector game obtained if the move $(-b,0)$ were not available.

\begin{defi}
    For any $L(b;x_1,y_1;x_2,y_2)$, let $\mathcal{SG}^*(x,y)=\begin{cases}
        \mathcal{SG}(x,y) & \mathcal{SG}(x,y)=0,1,2\\
        2 & \mathcal{SG}(x,y)=3\\
    \end{cases}$
\end{defi}

\begin{defi}\label{def: diagonal periodicity}
    Consider the array of \textsc{nim}-values of a particular $L(b;x_1,y_1;x_2,y_2)$.  We say it has \emph{diagonal periodicity} if, for all $(x,y)$ with $y\geq y_1+y_2$, we have $\mathcal{SG}^*(x,y)=\mathcal{SG}^*(x+x_1+x_2,y-y_1-y_2)$.  Similarly, we say it has diagonal periodicity after some $x_0$ if the condition holds for all $x\geq x_0$.
\end{defi}

In other words, $L(b;x_1,y_1;x_2,y_2)$ is diagonally periodic if applying both of the moves $(-x_1,y_1)$ and $(-x_2,y_2)$ does not change the \textsc{nim}-value, except possibly switching $2$s and $3$s.  In this case, Corollary \ref{corollary: 2s vs 3s} shows us when these switches must occur: between alternating blocks of $b$ rows.  Refer to Figure \ref{fig:L(2;3,0;1,1)}, where the `diagonal direction' in question is $(4,-1)$.  Note that we are now mentally partitioning \textsc{nim}-values as $\{0\mid 1\mid 2,3\}$.  We will prove a much stronger version of the following in Section \ref{section - eventual diagonal periodicity}, but for now, we lay down some necessary groundwork.

\begin{prop}\label{prop: diagonal periodicity if a>2c}
    If $x_1\geq 2x_2$ or $x_2\geq 2x_1$, then $L(b;x_1,y_1;x_2,y_2)$ is diagonally periodic.
\end{prop}

In order to prove this proposition, we need another structural fact about Lengyel transfer games.  That is, they follow a pattern where the columns alternate between blocks without $2$s and $3$s, and blocks which might contain them. 

\begin{defi}
    Consider the array of \textsc{nim}-values of $L(b;x_1,y_1;x_2,y_2)$.  A column is called a \emph{$2$-column} if it contains some $2$s (equivalently, some $3$s).
\end{defi}

\begin{lem}[The $2$-block Rule]\label{lemma: ac rule}
Suppose $x_2\leq x_1$.  Then the columns of the \textsc{nim}-array of $L(b;x_1,y_1;x_2,y_2)$ numbered $0,1,\ldots, x_1-1 \bmod (x_1+x_2)$ are not $2$-columns.    
\end{lem}

In this case, we call the columns numbered $x_1, x_1+1,\ldots, x_1+x_2-1\bmod (x_1+x_2)$ \emph{potential $2$-columns}. We then call a consecutive block of $x_2$ potential $2$-columns a \emph{$2$-block}.
\begin{figure}[h]
    \centering
    \begin{tikzpicture}
\foreach \i in {0,5,8,13,16}{
\draw[thick,gray] (\i,0) -- (\i,0.9);}
\draw[thick,gray] (0,0.9) -- (16,0.9);
\node at (2.5,0.5) {No $2$s or $3$s};
\draw [decorate,
    decoration = {calligraphic brace}] (0.05,1) --  (4.95,1)
    node[pos=0.5,above=6pt,black]{$x_1$ columns};
\node at (6.5,0.5) {Maybe $2$s and $3$s};
\draw [decorate,
    decoration = {calligraphic brace}] (5.05,1) --  (7.95,1)
    node[pos=0.5,above=6pt,black]{$x_2$ columns};
\node at (10.5,0.5) {No $2$s or $3$s};
\draw [decorate,
    decoration = {calligraphic brace}] (8.05,1) --  (12.95,1)
    node[pos=0.5,above=6pt,black]{$x_1$ columns};
\node at (14.5,0.5) {Maybe $2$s and $3$s};
\draw [decorate,
    decoration = {calligraphic brace}] (13.05,1) --  (15.95,1)
    node[pos=0.5,above=6pt,black]{$x_2$ columns};
\node at (16.3,0.5) {$\cdots$};
\end{tikzpicture}
\caption{The $2$-block Rule  (Lemma \ref{lemma: ac rule})}
    \label{fig:$2$-block rule}
\end{figure}
\begin{proof}
    By \textsc{nim}-periodicity, we may restrict our attention to the first $b$ rows.

    The positions in the first $x_1$ columns each then have at most $1$ option, so the result follows here.

    Now we show the result by induction for $x\geq x_1+x_2$ and $x=0,1,\ldots, x_1-1 \bmod (x_1+x_2)$. Indeed let $x$ be the smallest such value so that $\mathcal{SG}(x,y)=2$ for some $y$.  This means $\{\mathcal{SG}(x-x_1,y+y_1),\mathcal{SG}(x-x_2,y+y_2)\}=\{0,1\}$.  By the choice of $x$, $\mathcal{SG}(x-x_1-x_2,y+y_1+y_2)\in\{0,1\}$, but $(x-x_1-x_2,y+y_1+y_2)$ is an option of both $(x-x_1,y+y_1)$ and $(x-x_2,y+y_2)$, a contradiction.
\end{proof}


Next we show that we have diagonal periodicity once a particular configuration of $2$s and $3$s are absent.

\begin{defi}
    Suppose $\mathcal{SG}^*(x,y)=\mathcal{SG}^*(x',y')=2$, where $(x-x',y-y')=(x_1-x_2,y_2-y_1)$.  We call $(x,y)$ and $(x',y')$ a \emph{bad pair of $2$s}. 
\end{defi}

In this case, $(x,y)$ and $(x',y')$ share the option $(x-x_1,y+y_1)=(x'-x_2, y'+y_2)$.  In other words, $(x,y)$ and $(x',y')$ are the the red positions in Figure \ref{Figure - diagonal periodicity}.

\begin{lem}\label{lem-no bad pairs of twos implies diagonal periodicity}
If $L(b;x_1,y_1;x_2,y_2)$ has no bad pairs of $2$s for all $x\geq x_0$, then it has diagonal periodicity after $x_0+\max\{x_1,x_2\}$.
\end{lem}

\begin{proof}
    Consider $(x,y)$ so that $0\leq y-y_1-y_2<b$. In this case, $(x+x_1+x_2,y-y_1-y_2)$ has only two options, so we may apply the methods of the proof of Lemma \ref{lemma: refinement of Larsson}.  The proof of the remaining cases follows by \textsc{nim}-periodicity.  See Figure \ref{Figure - diagonal periodicity}.
    
    Suppose $x\geq x_0$. If $\mathcal{SG}(x,y)=0$, then $\mathcal{SG}(x+x_1,y-y_1)$ and $\mathcal{SG}(x+x_2,y-y_2)$ are both positive.  Hence $\mathcal{SG}(x+x_1+x_2,y-y_1-y_2)=0$.
    If $\mathcal{SG}(x,y)=1$, then $\mathcal{SG}(x+x_1,y-y_1), \mathcal{SG}(x+x_2,y-y_2)\in\{0,2,3\}$.  If neither of these were $0$, they would be a bad pair of $2$s.  Hence $\mathcal{SG}(x+x_1+x_2,y-y_1-y_2)=1$.  

    Now suppose $x\geq x_0+\max\{x_1,x_2\}$.   If $\mathcal{SG}(x,y)\in\{2,3\}$, then  $\mathcal{SG}(x,y-b)$, if it exists, is either $2$ or $3$, by \textsc{nim}-periodicity. Hence $\{\mathcal{SG}(x-x_1,y+y_1), \mathcal{SG}(x-x_2,y+y_2)\}=\{0,1\}$.  Then by the previous cases, $\{\mathcal{SG}(x+x_1,y-y_1), \mathcal{SG}(x+x_2,y-y_2)\}=\{0,1\}$, and so $\mathcal{SG}(x+x_1+x_2,y-y_1-y_2)\in\{2,3\}$
\end{proof}

\begin{proof}[Proof of Proposition \ref{prop: diagonal periodicity if a>2c}]
       Suppose $x_2\geq 2x_1$.  We show that there are no bad pairs of $2$s.  Indeed if $(x,y)$ and $(x',y')$ is a bad pair of $2$s, then $x-x'=x_2-x_1$ and so $x_1\leq x-x'\leq x_2$.  In other words, $x-x'\geq x_1\bmod (x_1+x_2)$, violating Lemma \ref{lemma: ac rule}.   
\end{proof}

\begin{figure}[h]
    \centering
    \begin{tikzpicture}[scale=1] 
    \node[black] (A)  at (0,0) {$(x,y)$};
    \node[red] (B) at (4,0.5) {$(x+x_1, y-y_1)$};
    \node[red] (C) at (0.6,1.2) {$(x+x_2, y-y_2)$};
    \node[black] (D) at (4.6,1.7) {$(x+x_1+x_2, y-y_1-y_2)$};
    \draw [->] (A) -- (B);
    \draw [->] (A) -- (C);
    \draw [->] (B) -- (D);
    \draw [->] (C) -- (D);
    \end{tikzpicture}
    \caption{Proof of Proposition \ref{lem-no bad pairs of twos implies diagonal periodicity}}
    \label{Figure - diagonal periodicity}
\end{figure}

Finally, we show that, as we move from one $2$-block to the next, there are no new $2$s or $3$s, and that the $2$s and $3$s which remain follow diagonal periodicity. 

\begin{lem}\label{bad pair of 2s less than or equal to}
If $\mathcal{SG}^*(x+x_1+x_2,y-y_1-y_2)=2$, then $\mathcal{SG}^*(x,y)=2$.
\end{lem}

\begin{proof}
    Following Figure \ref{Figure - diagonal periodicity}, let $\mathcal{SG}^*(x+x_1+x_2,y-y_1-y_2)=2$.    Then, by \textsc{nim}-periodicity, $\mathcal{SG}^*(x+x_1+x_2,y-y_1-y_2-b)$, if it exists, is $2$.  Hence
$\{\mathcal{SG}(x+x_1,y-y_1),\mathcal{SG}(x+x_2,y-y_2)\}=\{0,1\}$.  Thus $\mathcal{SG}^*(x,y)=2$.
\end{proof}

\section{Proof of Theorem \ref{Main Theorem 1}: The Period of $L(b;x_1,0;1, 1)$}\label{Section - Proof of Main Theorem 1}

\begin{prop}
    \label{corollary: $L(a,b,1,1)$ has diagonal periodicity}
    Assume that $(b,x_1)\neq (1,1)$. The game $L(b;x_1,0;1,1)$ has diagonal periodicity.
\end{prop}

Note that Figure \ref{fig:L(1;1,0;1,1)} shows that the conclusion fails if  $(b,x_1)= (1,1)$.  

\begin{proof}
    Proposition \ref{prop: diagonal periodicity if a>2c} implies the result if $x_1\geq 2$. 
    
    Assume $x_1=1$ and $b\geq 2$.  Lemma \ref{lemma: ac rule} says that exactly the odd numbered columns are potential $2$-columns.      
    Then by inspection, we see that the first $2$-block does not contain any bad pairs of $2$s (See Figure \ref{L(3,1,0,1,1)}, for example), so by Lemma \ref{bad pair of 2s less than or equal to}, there are not any at all.  Then Lemma \ref{lem-no bad pairs of twos implies diagonal periodicity} implies the result.
\end{proof}
\begin{figure}[h]
    \centering
    $\begin{tblr}{cccccccccccc}
         0 & 1 & 0 & 1 & 0 & 2 & 1 & 0 & 1 & 0 & 1 & 2\\
         0 & 1 & 0 & 2 & 1 & 0 & 1 & 0 & 1 & 2 & 0 & 1\\
         0 & 2 & 1 & 0 & 1 & 0 & 1 & 2 & 0 & 1 & 0 & 1 \\ \hline[dashed]
         1 & 0 & 1 & 0 & 1 & 3 & 0 & 1 & 0 & 1 & 0 & 3\\
         1 & 0 & 1 & 3 & 0 & 1 & 0 & 1 & 0 & 3 & 1 & 0\\
         1 & 3 & 0 & 1 & 0 & 1 & 0 & 3 & 1 & 0 & 1 & 0
    \end{tblr}$
    \caption{A period of the \textsc{nim}-values for $L(3;1,0;1,1)$}
    \label{L(3,1,0,1,1)}
\end{figure}

Next Lemma \ref{lemma: ac rule} says that the first $x_1$ columns are not $2$-columns.  The cases (b) and (c) of Theorem \ref{Main Theorem 1} are determined by whether the $(x_1+1)$th is a $2$-column.

\begin{lem}\label{2 columns in the proof of MT1}
    The $(x_1+1)$th column of the game $L(b;x_1,0;1,1)$ is a $2$-column, except when $b$ is even and $x_1$ is an odd multiple of $b$.  
\end{lem}
\begin{proof}
    By \textsc{nim}-periodicity, we just need to check whether any of the first $b$ entries of this column have \textsc{nim}-value $2$.  Note that each of these positions only have $2$ options.

    The first $x_1$ columns of $L(b;x_1,0;1,1)$ are identical to that of $L(b;1,1)$.  In particular, the first $b$ entries of the first column are all $0$.
    
    Now consider the $(x_1+1)$th column of $L(b;1,1)$. Observing Figure \ref{fig:two-move case}, we see that this column contains some $0$s in the first $b$ entries, except in the case that $b$ is even and $x_1$ is an odd multiple of $b$.

    In the first case, consider one of these $0$ entries, $(x_1,y)$.  Then $\mathcal{SG}(x_1-1,y+1)$ must be $1$, and $\mathcal{SG}(0,y)=0$.  Hence, in the game $L(b;x_1,0;1,1)$, we have $\mathcal{SG}(x,y)=2$.

    In the latter case, when $\mathcal{SG}(x_1,y)=1$ for all $0\leq y
    \leq b-1$, we must have $\mathcal{SG}(x_1-1,y+1)=0$ for all $0\leq y
    \leq b-1$.  Since we also have $\mathcal{SG}(0,y)=0$ for all $0\leq y
    \leq b-1$, then in the game $L(b;x_1,0;1,1)$, we have $\mathcal{SG}(x_1,y)=1$ for all $0\leq y
    \leq b-1$.
    %
    %
\end{proof}

\begin{proof}[Proof of Theorem \ref{Main Theorem 1}]  

Again, the proof of part (a) is essentially Figure \ref{fig:L(1;1,0;1,1)}.  So assume $(b,x_1)\neq (1,1)$.

Suppose first that it is not the case that $b$ is even and $x_1$ is an odd multiple of $b$.  By Lemma \ref{2 columns in the proof of MT1}, the $(x_1+1)$th column is a $2$ column.  Next, by Proposition \ref{corollary: $L(a,b,1,1)$ has diagonal periodicity}, the first $x_1+1$ columns repeat with an upward shift of $1$ unit.  We then know that it takes $2b$ such shifts to repeat, since the first column of the period is precisely $b$ $0$s followed by $b$ $1$s.  Hence the horizontal component of the period is $2b(x_1+1)$

In the case that $b$ is even and $x_1$ is an odd multiple of $b$, Lemma \ref{2 columns in the proof of MT1} states that the $(x_1+1)$th column does not contain any $2$s or $3$s.  Then Proposition \ref{corollary: $L(a,b,1,1)$ has diagonal periodicity} implies that no later columns do either.
The only way for this to happen is if, for any such position, all options have the same \textsc{nim}-value.  In other words, the values of this game are same as that of $L(b;1,1)$, which, by Corollary \ref{cor: period of L(b,1,1)}, has a period whose horizontal component is $2b$ when $b$ is even. 
\end{proof}

\section{Long preperiods}\label{section - long preperiods}

In this section, we examine a set of examples of $L(b;x_1,0;x_2,y_2)$ with arbitrarily long preperiod.  One set has period 2, showing that the preperiod can be arbitrarily longer than the period.  The other set has period $g(b,x_1,0,x_2,y_2)$, showing that we have arbitrarily long preperiod when the period is maximal.

\begin{prop}\label{prop - Sporadic1}  The following games have the following horizontal preperiods and periods:
\begin{center}
\begin{tabular}{c|l|c|l}
   \; &\; &preperiod& period\\
    \hline
    i&$L(b; 1, 0; 1, b-1)$, $b$ odd & $b-1$ & $2$\\\hline
    ii&$L(b; 1, 0; 1, b-1)$, $b$ even & $b-2$ & $4b$ \\\hline
    iii&$L(b; 1, 0; 1, b+1)$,  $b$ odd & $b-1$ & $2$ \\\hline
    iv&$L(b; 1, 0; 1, b+1)$, $b$ even & $b-2$ & $4b$     \\\hline
\end{tabular}
\end{center}
\end{prop}

\begin{proof}
    Figure \ref{L(b;1,0;1,b-1)} shows the preperiod and the first two columns of the period  for each case of $b=7$ and $b=8$.  After careful inspection, one finds that the pattern persists for other values of $b$.
\end{proof}

We have highlighted the bad pairs of $2$s in Figure \ref{L(b;1,0;1,b-1)} so the reader can observe the exact manner in which they disappear over the course of the preperiod.

\begin{figure}[h]
    \centering
    $\begin{tblr}{cccccc|cc}
         0 & 1 & 0 & 1 & 0 & 1 & 0 & 1 \\
         0 & 2 & 1 & 0 & 1 & 0 & 1 & 0 \\
         0 & \SetCell{cyan}2 & 0 & 1 & 0 & 1 & 0 & 1 \\
         0 & \SetCell{cyan}2 & 0 & 2 & 1 & 0 & 1 & 0 \\
         0 & \SetCell{cyan}2 & 0 & \SetCell{cyan}2 & 0 & 1 & 0 & 1 \\
         0 & \SetCell{cyan}2 & 0 & \SetCell{cyan}2 & 0 & 2 & 1 & 0 \\
         0 & \SetCell{cyan}2 & 0 & \SetCell{cyan}2 & 0 & \SetCell{cyan}2 & 0 & 1 \\
         \hline[dashed]
         1 & 0 & 1 & 0 & 1 & 0 & 1 & 0 \\
         1 & \SetCell{yellow}3 & 0 & 1 & 0 & 1 & 0 & 1 \\
         1 & \SetCell{yellow}3 & 1 & 0 & 1 & 0 & 1 & 0 \\
         1 & \SetCell{yellow}3 & 1 & \SetCell{yellow}3 & 0 & 1 & 0 & 1 \\
         1 & \SetCell{yellow}3 & 1 & \SetCell{yellow}3 & 1 & 0 & 1 & 0 \\
         1 & \SetCell{yellow}3 & 1 & \SetCell{yellow}3 & 1 & \SetCell{yellow}3 & 0 & 1 \\
         1 & 3 & 1 & 3 & 1 & 3 & 1 & 0 \\    \end{tblr}$
         \qquad
         $\begin{tblr}{cccccc|cc}
         0 & 1 & 0 & 1 & 0 & 1 & 0 & 1 \\
         0 & 2 & 1 & 0 & 1 & 0 & 1 & 0 \\
         0 & \SetCell{cyan}2 & 0 & 1 & 0 & 1 & 0 & 1 \\
         0 & \SetCell{cyan}2 & 0 & 2 & 1 & 0 & 1 & 0 \\
         0 & \SetCell{cyan}2 & 0 & \SetCell{cyan}2 & 0 & 1 & 0 & 1 \\
         0 & \SetCell{cyan}2 & 0 & \SetCell{cyan}2 & 0 & 2 & 1 & 0 \\
         0 & \SetCell{cyan}2 & 0 & \SetCell{cyan}2 & 0 & \SetCell{cyan}2 & 0 & 1 \\
         0 & \SetCell{cyan}2 & 0 & \SetCell{cyan}2 & 0 & \SetCell{cyan}2 & 0 & 2\\
         \hline[dashed]
         1 & 0 & 1 & 0 & 1 & 0 & 1 & 0 \\
         1 & \SetCell{yellow}3 & 0 & 1 & 0 & 1 & 0 & 1 \\
         1 & \SetCell{yellow}3 & 1 & 0 & 1 & 0 & 1 & 0 \\
         1 & \SetCell{yellow}3 & 1 & \SetCell{yellow}3 & 0 & 1 & 0 & 1 \\
         1 & \SetCell{yellow}3 & 1 & \SetCell{yellow}3 & 1 & 0 & 1 & 0 \\
         1 & \SetCell{yellow}3 & 1 & \SetCell{yellow}3 & 1 & \SetCell{yellow}3 & 0 & 1 \\
         1 & \SetCell{yellow}3 & 1 & \SetCell{yellow}3 & 1 & \SetCell{yellow}3 & 1 & 0 \\
         1 & 3 & 1 & 3 & 1 & 3 & 1 & 3 \\
         \end{tblr}$

\bigskip
         
         $\begin{tblr}{cccccc|cc}
         0 & \SetCell{cyan}2 & 0 & \SetCell{cyan}2 & 0 & \SetCell{cyan}2 & 0 & 1 \\
         0 & \SetCell{cyan}2 & 0 & \SetCell{cyan}2 & 0 & 2 & 1 & 0 \\
        0 & \SetCell{cyan}2 & 0 & \SetCell{cyan}2 & 0 & 1 & 0 & 1 \\
        0 & \SetCell{cyan}2 & 0 & 2 & 1 & 0 & 1 & 0 \\
        0 & \SetCell{cyan}2 & 0 & 1 & 0 & 1 & 0 & 1 \\
         0 & 2 & 1 & 0 & 1 & 0 & 1 & 0 \\
         0 & 1 & 0 & 1 & 0 & 1 & 0 & 1 \\
         \hline[dashed]
         1 & 3 & 1 & 3 & 1 & 3 & 1 & 0 \\    
          1 & \SetCell{yellow}3 & 1 & \SetCell{yellow}3 & 1 & \SetCell{yellow}3 & 0 & 1 \\
         1 & \SetCell{yellow}3 & 1 & \SetCell{yellow}3 & 1 & 0 & 1 & 0 \\
         1 & \SetCell{yellow}3 & 1 & \SetCell{yellow}3 & 0 & 1 & 0 & 1 \\
         1 & \SetCell{yellow}3 & 1 & 0 & 1 & 0 & 1 & 0 \\
         1 & \SetCell{yellow}3 & 0 & 1 & 0 & 1 & 0 & 1 \\
          1 & 0 & 1 & 0 & 1 & 0 & 1 & 0 \\
         \end{tblr}$
         \qquad
         $\begin{tblr}{cccccc|cc}
         0 & \SetCell{cyan}2 & 0 & \SetCell{cyan}2 & 0 & \SetCell{cyan}2 & 0 & 2\\
         0 & \SetCell{cyan}2 & 0 & \SetCell{cyan}2 & 0 & \SetCell{cyan}2 & 0 & 1 \\
        0 & \SetCell{cyan}2 & 0 & \SetCell{cyan}2 & 0 & 2 & 1 & 0 \\
         0 & \SetCell{cyan}2 & 0 & \SetCell{cyan}2 & 0 & 1 & 0 & 1 \\
         0 & \SetCell{cyan}2 & 0 & 2 & 1 & 0 & 1 & 0 \\
         0 & \SetCell{cyan}2 & 0 & 1 & 0 & 1 & 0 & 1 \\
         0 & 2 & 1 & 0 & 1 & 0 & 1 & 0 \\
         0 & 1 & 0 & 1 & 0 & 1 & 0 & 1 \\
         \hline[dashed]
         1 & 3 & 1 & 3 & 1 & 3 & 1 & 3 \\
1 & \SetCell{yellow}3 & 1 & \SetCell{yellow}3 & 1 & \SetCell{yellow}3 & 1 & 0 \\
1 & \SetCell{yellow}3 & 1 & \SetCell{yellow}3 & 1 & \SetCell{yellow}3 & 0 & 1 \\
1 & \SetCell{yellow}3 & 1 & \SetCell{yellow}3 & 1 & 0 & 1 & 0 \\
 1 & \SetCell{yellow}3 & 1 & \SetCell{yellow}3 & 0 & 1 & 0 & 1 \\
1 & \SetCell{yellow}3 & 1 & 0 & 1 & 0 & 1 & 0 \\
        1 & \SetCell{yellow}3 & 0 & 1 & 0 & 1 & 0 & 1 \\
         1 & 0 & 1 & 0 & 1 & 0 & 1 & 0 \\
         \end{tblr}$
    \caption{The preperiod of the \textsc{nim}-values for $L(7;1,0;1,6)$, $L(8;1,0;1,7)$, $L(7;1,0;1,8)$, and $L(8;1,0;1,9)$, together with the first two $2$-blocks of the initial period, showing diagonal periodicity.  The bad pairs of $2$s are highlighted.}
    \label{L(b;1,0;1,b-1)}
\end{figure}

\section{Simplifying assumptions}\label{Section - simplifying assumptions}

We work toward establishing a coherent conjecture about the horizontal period of $L(b;x_1,0;x_2,y_2)$.  To this end, we establish some simplifying assumptions.  Note that the results in this and the next section do apply to the more general case where $y_1$ is not necessarily $0$.

\begin{itemize}
    \item If $y_2=0$, then $L(b;x_1,0;x_2,y_2)$ is isomorphic to the one-dimensional, two-move subtraction game defined by $x_1$ and $x_2$.  In this case, apply Lemma \ref{Larssons lemma} to the first $x_1+x_2$ values.
    \item If $y_1\geq 2b$ or $y_2\geq 2b$, then apply Corollary \ref{cor: y can be reduced mod 2b}.
    \item If $\gcd\{x_1,x_2\}>1$ or if  $\gcd\{b,y_1,y_2\}>1$, then apply Corollary \ref{lemma:dilation}.
    \end{itemize}

Our last observation is that sometimes one of the vectors defining a game is redundant, and we can thus can be eliminated.  Then the \textsc{nim}-values can be computed via Theorem \ref{Theorem: 2 vector case}.

\begin{lem}[Vector Elimination]\label{lemma-vector elimination}
    Consider $L(b;x_1,y_1;x_2,y_2)$.  Suppose that either
    \begin{itemize}
        \item There is an odd $k$ so that $x_2=kx_1$ and $y_2 = ky_1 \bmod 2b$; or
        \item There is an even $k$ so that $x_2=kx_1$ and $y_2 = ky_1+b \bmod 2b$
    \end{itemize}
    then $L(b;x_1,y_1;x_2,y_2)$ has the same \textsc{nim}-values as $L(b;x_1,y_1)$.  In particular, it has no $2$s or $3$s.
\end{lem}

\begin{proof}
In the first case, the result is clear when $x< x_2$.  Now consider some $(x,y)$ with $x\geq x_2$, and proceed by induction on $x$.  
Suppose further that $y<b$, as the result for $y\geq b$ will then follow by \textsc{nim}-periodicity.

In $L(b;x_1,y_1)$, since the only \textsc{nim}-values are $0$ and $1$, we have that $\mathcal{SG}(p)=\mathcal{SG}(q)\oplus 1$ whenever $q$ is an option of $p$.  Hence
\[\mathcal{SG}(x,y), \mathcal{SG}(x-x_1,y+y_1), \mathcal{SG}(x-2x_1,y+2y_1), \ldots, \mathcal{SG}(x-kx_1,y+ky_1)\]
is an alternating $0,1$-sequence in $L(b;x_1,y_1)$.  By induction, \[\mathcal{SG}(x-x_1,y+y_1), \mathcal{SG}(x-2x_1,y+2y_1), \ldots, \mathcal{SG}(x-kx_1,y+ky_1)\]
is the same alternating $0,1$-sequence in $L(b;x_1,y_1;x_2,y_2)$.  In particular, since $k$ is odd, $\mathcal{SG}(x-x_1,y+y_1)=\mathcal{SG}(x-kx_1,y+ky_1)$.
But then in $L(b;x_1,y_1;x_2,y_2)$,
\begin{align*}
\mathcal{SG}(x,y)&=\mex\{\mathcal{SG}(x-x_1,y+y_1), \mathcal{SG}(x-x_2,y+y_2)\}\\
&=\mex\{\mathcal{SG}(x-x_1,y+y_1), \mathcal{SG}(x-kx_1,y+ky_1)\}\\
&=\mex\{\mathcal{SG}(x-x_1,y+y_1)\}\\
&=\mathcal{SG}(x-x_1,y+y_1)\oplus 1
\end{align*}
which is the same value as in $L(b;x_1,y_1)$.

In the second case, the proof is similar, except that since $k$ is even, we have $\mathcal{SG}(x-x_1,y+y_1)=\mathcal{SG}(x-kx_1,y+ky_1)\oplus1$.  Then in $L(b;x_1,y_1;x_2,y_2)$,
\begin{align*}
\mathcal{SG}(x,y)&=\mex\{\mathcal{SG}(x-x_1,y+y_1), \mathcal{SG}(x-x_2,y+y_2)\}\\
&=\mex\{\mathcal{SG}(x-x_1,y+y_1), \mathcal{SG}(x-kx_1,y+ky_1+b)\}\\
&=\mex\{\mathcal{SG}(x-x_1,y+y_1), \mathcal{SG}(x-kx_1,y+ky_1)\oplus 1\}&\text{by \textsc{nim}-periodicity}\\
&=\mex\{\mathcal{SG}(x-x_1,y+y_1)\}\\
&=\mathcal{SG}(x-x_1,y+y_1)\oplus 1
\end{align*}
which is the same value as in $L(b;x_1,y_1)$.
\end{proof}

\begin{cor}\label{cor - vector elimination}
    Consider $L(b;x_1,0;x_2,y_2)$, where $y_2<2b$.
    \begin{enumerate}[(a)]
        \item 
If \begin{itemize}
        \item $x_2$ is an odd multiple of $x_1$ and $y_2=0$; or
        \item $x_2$ is an even multiple of $x_1$ and $y_2=b$,
    \end{itemize}
    then $L(b;x_1,0;x_2,y_2)$ has the same \textsc{nim}-values as $L(b;x_1,0)$.
    \item If \begin{itemize}
        \item $x_1=kx_2$, where $k$ is odd, and $ky_2$ is an even multiple of $b$; or 
        \item $x_1=kx_2$, where $k$ is even, and $ky_2$ is an odd multiple of $b$,
    \end{itemize}
    then $L(b;x_1,0;x_2,y_2)$ has the same \textsc{nim}-values as $L(b;x_2,y_2)$.    \end{enumerate}
\end{cor}
\noindent Note that the last case is precisely what happens in Theorem \ref{Main Theorem 1}(b).  \textbf{For the next two sections, assume the following:}
\begin{itemize}
    \item $0\leq y_1<2b$
    \item $0<y_2<2b$
    \item $\gcd\{x_1,x_2\}=1$
    \item $\gcd\{b,y_1,y_2\}=1$ (or $\gcd\{b,y_2\}=1$ in the case $y_1=0$)
    \item Lemma \ref{lemma-vector elimination} (or Corollary \ref{cor - vector elimination} in the case $y_1=0$) does not apply
\end{itemize}

\section{Eventual diagonal periodicity}\label{section - eventual diagonal periodicity}

In this section, we show that with two exceptions, $L(b;x_1,y_1;x_2,y_2)$  has eventual diagonal periodicity due to the eventual disappearance of bad pairs of $2$s.  In one of the exceptional cases, we have diagonal periodicity anyway.  In the remaining case, we do not have diagonal periodicity, but the chart of \textsc{nim}-values is trivial.  

By a \emph{chain of bad pairs of $2$s}, we mean a sequence $(x,y), (x+m,y+n), (x+2m,y+2n), \ldots$ where each adjacent pair is a bad pair of $2$s.



    

\begin{lem}\label{lem: stable chains observe diagonal periodicity}
    Suppose $\mathcal{SG}^*(x,y)=\mathcal{SG}^*(x+x_1+x_2,y-y_1-y_2)=2$, and that $(x,y)$ is not part of two bad pairs of $2$s.  Then $\mathcal{SG}(x-x_1,y+y_1)=\mathcal{SG}(x+x_2,y-y_2)$ and $\mathcal{SG}(x-x_2,y+y_2)=\mathcal{SG}(x+x_1,y-y_1)$.
\end{lem}

\begin{figure}[h]
    \centering
    \begin{tikzpicture}[scale=0.7] 
    \node[black] (A)  at (0,0) {$A_1$};
    \node[black] (C) at (0.6,1.1) {$2/3$};
    \node[black] (E) at (-3.4,0.7) {$A_0$};
    \node[black] (G) at (-4,-0.4) {$2/3$};
    \node[black] (H) at (-0.6,-1.1) {$D$};
    \node[black] (I)  at (-4.6,-1.5) {$B_1$};
    \node[black] (J) at (-7.4,0.3) {$C$};
    \node[black] (K) at (-8,-0.8) {$B_0$};
    \draw [->] (A) -- (C);
    \draw [->] (G) -- (E);
    \draw [->] (E) -- (C);
    \draw [->] (G) -- (A);
    \draw [->] (H) -- (A);
    \draw [->] (I) -- (G);
    \draw [->] (I) -- (H);
    \draw [->] (J) -- (E);
    \draw [->] (K) -- (J);
    \draw [->] (K) -- (G);
    \end{tikzpicture}
    \caption{Proof of Lemma \ref{lem: stable chains observe diagonal periodicity}.  The arrows represent the vectors $(x_1,-y_1)$ and $(x_2,-y_2)$.  }
    \label{fig: Proof lemma about chains of bad pairs of 2s}
\end{figure}

\begin{proof}
    Observe Figure \ref{fig: Proof lemma about chains of bad pairs of 2s}, where $(x,y)$ and $(x+x_1+x_2,y-y_1-y_2)$ are the positions marked $2/3$, and either $\mathcal{SG}^*(C)\neq 2$ or $\mathcal{SG}^*(D)\neq 2$.  Suppose $\mathcal{SG}^*(C)\neq 2$. By \textsc{nim}-periodicity, $(x,y-b)$ and $(x+x_1+x_2,y-y_1-y_2-b)$, if they exist, also each have value $2$ or $3$.  Thus $\{A_0, A_1\}=\{B_0, B_1\}=\{0,1\}$.  Without loss of generality, $A_0=0$ and $A_1=1$.  Then $\mathcal{C}=1$, so $B_0=0$ and $B_1=1$.
\end{proof}

\begin{prop}\label{prop: bad pairs eventually disappear}
    Suppose  $y_1+y_2\neq 2b$.  Suppose further that either (i) $x_1\neq x_2$, or (ii) $x_1=x_2=1$ and $\lvert y_1-y_2 \rvert\neq b$.
    Then any bad pairs of $2$s in $L(b;x_1,y_1;x_2,y_2)$ eventually disappear.
\end{prop}
    
\begin{proof} Assume toward a contradiction that the bad pairs of $2$s do not eventually disappear.  
\begin{claim}
    For any $k$, we can find a chain of at least $k$ bad pairs of $2$s.
\end{claim}
\begin{proof}[proof of claim]
    Suppose not.  Then by Lemma \ref{bad pair of 2s less than or equal to}, there is some column $x_0$, after which all chains of bad pairs of $2$s are of the minimal length, say $k_0>1$.
    
    By Lemma \ref{lem: stable chains observe diagonal periodicity}, along with the fact that $y_1+y_2\neq 2b$, we can select the chain far enough to the right, so that the position $(x,y)$ (corresponding to $A_0$ or $A_1$ in Figure \ref{fig: Proof lemma about chains of bad pairs of 2s})
    satisfies $\mathcal{SG}^*(x-x_1,y+y_1)=\mathcal{SG}^*(x-x_2,y+y_2)=2$, and either
    \begin{itemize}
        \item $\mathcal{SG}(x,y)=1$ and $0\leq y<b$; or
        \item $\mathcal{SG}(x,y)=0$ and $b\leq y<2b$.
    \end{itemize}  
    
    In either case, \textsc{nim}-periodicity gives us a node $(x,y')$ with $y'<b$, $\mathcal{SG}(x,y')=1$, and both options of $(x,y')$ having value $2$ or $3$, a contradiction.
\end{proof}

In Case (i), choose a large enough $k$ to violate Lemma \ref{lemma: ac rule}.

In Case (ii),  let $(x,y+jm)$, where $m=\lvert y_1-y_2\rvert$, for $0\leq j\leq 2b$ be a chain of bad pairs of $2$s.  Letting $x'=x-x_1, y'=y+y_1$, we obtain a sequence of positions $(x',y'+jm)$, whose values form an alternating sequence of $0$s and $1$s.  Since $m\neq b$, there must be some $j$ so that $y'+jm=2bl+i$ for some $0\leq i<b$ and $\mathcal{SG}(x',y'+jm)=1$.  By \textsc{nim}-periodicity, $\mathcal{SG}(x',i)=\mathcal{SG}(x',y'+jm-2bl)=1$.  However,  since $i<b$,  $(x',i)$ only has options from the chain of bad pairs of $2$s, a contradiction.
\end{proof}



Let's now consider the exceptions to the above proposition.  

\begin{prop}
    Suppose $y_1+y_2=2b$.  
    Then $L(b;x_1,y_1;x_2,y_2)$ has diagonal periodicity.  In particular, it has no preperiod and the horizontal period is $x_1+x_2 \;(=g(b,x_1,y_1,x_2,y_2))$.  
\end{prop}

\begin{figure}[h]
    \centering
    $\begin{tblr}{ccccccccc}
         0&0&0&0&1&2&2&2&2\\
         0&0&0&0&1&2&2&2&2\\
         0&0&0&0&0&0&0&0&2\\
         0&0&0&0&0&2&2&2&1\\
         0&0&0&0&0&2&2&2&2\\
         1&1&1&1&0&3&3&3&3\\
         1&1&1&1&0&3&3&3&3\\
         1&1&1&1&1&1&1&1&3\\
         1&1&1&1&1&3&3&3&0\\
         1&1&1&1&1&3&3&3&3
    \end{tblr}$
    \qquad\qquad
    $\begin{tblr}{cc}
         0&0\\
         0&2\\
         0&0\\
         1&1\\
         1&3\\
         1&1
    \end{tblr}$\qquad\qquad
    $\begin{tblr}{cc}
         0&2\\
         0&2\\
         0&2\\
         1&3\\
         1&3\\
         1&3
    \end{tblr}$
    \caption{One period each of the \textsc{nim}-values of $L(5;5,4;7,3)$, $L(3;1,2;1,4)$, and $L(3;1,2;1,5)$}
    \label{fig:$L(5;5,4;7,3)$, $L(3;1,2;1,4)$, and $L(3;1,2;1,5)$}
\end{figure}

See for example Figure \ref{fig:$L(5;5,4;7,3)$, $L(3;1,2;1,4)$, and $L(3;1,2;1,5)$} for the values of $L(5;5,4;7,3)$ and $L(3;1,2;1,4)$.  

\begin{proof}
    We will show that $\mathcal{SG}^*(x,y)=\mathcal{SG}^*(x+x_1+x_2,y-2b)$ in the case $2b\leq y<3b$. All other cases follow by \textsc{nim}-periodicity.

    Suppose $\mathcal{SG}^*(x,y)=0$.  Then $\mathcal{SG}^*(x+x_1,y-y_1), \mathcal{SG}^*(x+x_2,y-y_2)\neq 0$, so $\mathcal{SG}^*(x+x_1+x_2,y-2b)=0$.

    Next, suppose $\mathcal{SG}^*(x,y)=1$.  By Lemma \ref{bad pair of 2s less than or equal to}, $\mathcal{SG}^*(x+x_1+x_2,y-2b)\neq 2$, so assume toward a contradiction that $\mathcal{SG}^*(x+x_1+x_2,y-2b)=0$.  This means $\mathcal{SG}^*(x+x_1,y-y_1),\mathcal{SG}^*(x+x_2,y-y_2)=2$.  Hence by Lemma \ref{bad pair of 2s less than or equal to} again, $\mathcal{SG}^*(x-x_1,y+y_1)$ and $\mathcal{SG}^*(x-x_2,y+y_2)$, when they exist, are $2$.  However, by \textsc{nim}-periodicity, this means that $\mathcal{SG}^*(x-x_1,y+y_1-2b)$ and $\mathcal{SG}^*(x-x_2,y+y_2-2b)$, when they exist, are also $2$.  But these are the only possible options of $(x,y-2b)$, contradicting that $\mathcal{SG}^*(x,y-2b)=\mathcal{SG}^*(x,y)=1$.

    Finally, suppose $\mathcal{SG}^*(x,y)=2$. Then since $\mathcal{SG}^*(x,y-b)=2$, we must have $\{\mathcal{SG}^*(x-x_1,y+y_1), \mathcal{SG}^*(x-x_2,y+y_2)\}=\{0,1\}$.  Then by the previous cases,
     $\{\mathcal{SG}^*(x+x_1,y-y_1), \mathcal{SG}^*(x+x_2,y-y_2)\}=\{0,1\}$.  Hence $\mathcal{SG}^*(x+x_1+x_2,y-2b)=2$.
\end{proof}

\begin{prop}\label{prop - the trivial case generalized}
    Suppose $x_1=x_2=1$ and $\lvert y_1-y_2 \rvert=b$.  Then the \textsc{nim}-values are as in the third part of Figure \ref{fig:$L(5;5,4;7,3)$, $L(3;1,2;1,4)$, and $L(3;1,2;1,5)$}. In particular, there is no preperiod, and the horizontal period is $2$. 
\end{prop}

\begin{proof}
    See Figure \ref{fig:$L(5;5,4;7,3)$, $L(3;1,2;1,4)$, and $L(3;1,2;1,5)$} for the case of $L(3;1,2;1,5)$.  All other cases are similar.
\end{proof}

Note that we do not have diagonal periodicity, unless we additionally require $y_1+y_2=2b$.  This happens precisely when $\{y_1,y_2\}=\{\frac{b}{2},\frac{3b}{2}\}$.  Theorem \ref{Main Theorem 1}(a) is a special case of this. See also Figure \ref{fig:L(1;1,0;1,1)}.  

We are now ready to prove our second main theorem.

\begin{proof}[proof of Theorem \ref{conjecture - big one}]
By the preceding propositions, the only case when $L(b;x_1,y_1;x_2,y_2)$ does not have diagonal periodicity is precisely when $x_1=x_2$, $\lvert y_1-y_2 \rvert=b\bmod 2b$, and $y_1+y_2\neq 2b\bmod 2b$.  However, combining Proposition \ref{prop - the trivial case generalized} with Corollary \ref{lemma:dilation} yields a horizontal period of $(x_1+x_2)$, which is a factor of $g(b,x_1,y_1,x_2,y_2)$.

So suppose $L(b;x_1,y_1;x_2,y_2)$ has diagonal periodicity, and that $2b=k\cdot\gcd\{2b,y_1+y_2\}$.  In other words, $k(y_1+y_2)=0\bmod 2b$.  
Then by \textsc{nim}-periodicity and diagonal periodicity, 
\[\mathcal{SG}^*(x,y)=\mathcal{SG}^*(x+k(x_1+x_2),y-k(y_1+y_2))=\mathcal{SG}^*(x+k(x_1+x_2),y).\]
Finally, Corollary \ref{corollary: 2s vs 3s} implies $\mathcal{SG}(x,y)=\mathcal{SG}(x+k(x_1+x_2),y)$.
\end{proof}

\begin{rem}
    This proof shows how \textsc{nim}-periodicity and diagonal periodicity together imply horizontal periodicity, with a (maximum possible) period specified by the same parameters. 

    For the converse, suppose we assume \textsc{nim}-periodicity in $(0,-b)$ and a horizontal period of $p$.  Then we obtain periodicity in $(p,-2b)$.  In other words, we obtain a kind of diagonal periodicity, but not necessarily in the direction specified by the sum of the moves $(x_1,y_1)$ and $(x_2,y_2)$.  
\end{rem}

\section{The horizontal period of $L(b;x_1,0;x_2,y_2)$}\label{section - $L(b;x_1,0;x_2,y_2)$}

\subsection{A few more sporadic cases}


\begin{prop}\label{prop - sporadic2}
    Suppose $x_1$ is odd and $x_2=x_1\pm 1$. Then the horizontal period of $L(1;x_1,0;x_2,1)$ is $2$.
\end{prop}

\begin{figure}[h]
    
    $\begin{array}{>{\columncolor{yellow!20}}c>{\columncolor{yellow!20}}c>{\columncolor{yellow!20}}c>{\columncolor{yellow!20}}c>{\columncolor{yellow!20}}c>{\columncolor{yellow!20}}c>{\columncolor{yellow!20}}c
    >{\columncolor{blue!20}}c
    >{\columncolor{blue!20}}c
    >{\columncolor{blue!20}}c
    >{\columncolor{blue!20}}c
    >{\columncolor{blue!20}}c
    >{\columncolor{blue!20}}c
    |
    >{\columncolor{yellow!20}}c>{\columncolor{yellow!20}}c>{\columncolor{yellow!20}}c>{\columncolor{yellow!20}}c>{\columncolor{yellow!20}}c>{\columncolor{yellow!20}}c>{\columncolor{yellow!20}}c
    >{\columncolor{blue!20}}c
    >{\columncolor{blue!20}}c
    >{\columncolor{blue!20}}c
    >{\columncolor{blue!20}}c
    >{\columncolor{blue!20}}c
    >{\columncolor{blue!20}}c
    |}
         0& 0& 0& 0& 0& 0& 0& 2& 2& 2& 2& 2& 2& 1& 0& 0& 0& 0& 0& 1& 0& 2& 2& 2& 2& 1 \\
          1& 1& 1& 1& 1& 1& 1& 3& 3& 3& 3& 3& 3& 0& 1& 1& 1& 1& 1& 0& 1& 3& 3& 3& 3& 0 
\end{array}$\medskip

         $\begin{array}{>{\columncolor{yellow!20}}c>{\columncolor{yellow!20}}c>{\columncolor{yellow!20}}c>{\columncolor{yellow!20}}c>{\columncolor{yellow!20}}c>{\columncolor{yellow!20}}c>{\columncolor{yellow!20}}c
    >{\columncolor{blue!20}}c
    >{\columncolor{blue!20}}c
    >{\columncolor{blue!20}}c
    >{\columncolor{blue!20}}c
    >{\columncolor{blue!20}}c
    >{\columncolor{blue!20}}c
    |
    >{\columncolor{yellow!20}}c>{\columncolor{yellow!20}}c>{\columncolor{yellow!20}}c>{\columncolor{yellow!20}}c>{\columncolor{yellow!20}}c>{\columncolor{yellow!20}}c>{\columncolor{yellow!20}}c
    >{\columncolor{blue!20}}c
    >{\columncolor{blue!20}}c
    >{\columncolor{blue!20}}c
    >{\columncolor{blue!20}}c
    >{\columncolor{blue!20}}c
    >{\columncolor{blue!20}}c
    }
         0& 1& 0& 0& 0& 1& 0& 1& 0& 2& 2& 
         1& 0& 1&0 & 1&0 &1&0& 1&0& 1&0& 1&0& 1\\
         1& 0& 1& 1& 1& 0& 1& 0& 1& 3& 3& 0& 1& 0& 1&0 & 1&0 &1&0& 1&0& 1&0& 1&0
    \end{array}$ 

\bigskip\bigskip

    $\begin{array}{>{\columncolor{yellow!20}}c>{\columncolor{yellow!20}}c>{\columncolor{yellow!20}}c>{\columncolor{yellow!20}}c>{\columncolor{yellow!20}}c>{\columncolor{yellow!20}}c>{\columncolor{yellow!20}}c>{\columncolor{yellow!20}}c
    >{\columncolor{blue!20}}c
    >{\columncolor{blue!20}}c
    >{\columncolor{blue!20}}c
    >{\columncolor{blue!20}}c
    >{\columncolor{blue!20}}c
    >{\columncolor{blue!20}}c
    >{\columncolor{blue!20}}c
    |
    >{\columncolor{yellow!20}}c>{\columncolor{yellow!20}}c>{\columncolor{yellow!20}}c>{\columncolor{yellow!20}}c>{\columncolor{yellow!20}}c>{\columncolor{yellow!20}}c>{\columncolor{yellow!20}}c>{\columncolor{yellow!20}}c
    >{\columncolor{blue!20}}c
    >{\columncolor{blue!20}}c
    >{\columncolor{blue!20}}c
    >{\columncolor{blue!20}}c
    >{\columncolor{blue!20}}c
    >{\columncolor{blue!20}}c
    >{\columncolor{blue!20}}c
    |}
         0& 0& 0& 0& 0& 0& 0& 1& 2& 2& 2& 2& 2& 2& 0& 1& 0& 0& 0& 0& 0& 1& 0& 1& 2& 2& 2& 2& 0& 1 \\
         1& 1& 1& 1& 1& 1& 1& 0& 3& 3& 3& 3& 3& 3& 1& 0& 1& 1& 1& 1& 1& 0& 1& 0& 3& 3& 3& 3& 1& 0 
\end{array}$\medskip

         $\begin{array}{>{\columncolor{yellow!20}}c>{\columncolor{yellow!20}}c>{\columncolor{yellow!20}}c>{\columncolor{yellow!20}}c>{\columncolor{yellow!20}}c>{\columncolor{yellow!20}}c>{\columncolor{yellow!20}}c>{\columncolor{yellow!20}}c
    >{\columncolor{blue!20}}c
    >{\columncolor{blue!20}}c
    >{\columncolor{blue!20}}c
    >{\columncolor{blue!20}}c
    >{\columncolor{blue!20}}c
    >{\columncolor{blue!20}}c
    >{\columncolor{blue!20}}c
    |
    >{\columncolor{yellow!20}}c>{\columncolor{yellow!20}}c>{\columncolor{yellow!20}}c>{\columncolor{yellow!20}}c>{\columncolor{yellow!20}}c>{\columncolor{yellow!20}}c>{\columncolor{yellow!20}}c>{\columncolor{yellow!20}}c
    >{\columncolor{blue!20}}c
    >{\columncolor{blue!20}}c
    >{\columncolor{blue!20}}c
    >{\columncolor{blue!20}}c
    >{\columncolor{blue!20}}c
    >{\columncolor{blue!20}}c
    >{\columncolor{blue!20}}c
    }
         0& 1& 0& 0& 0& 1& 0& 1& 0& 1& 2& 2& 0& 1& 0& 1& 0& 1& 0& 1& 0& 1& 0& 1& 0& 1& 0& 1& 0& 1\\
         1& 0& 1& 1& 1& 0& 1& 0& 1& 0& 3& 3& 1& 0& 1& 0& 1& 0& 1& 0& 1& 0& 1& 0& 1& 0& 1& 0& 1& 0
    \end{array}$



    \caption{The first $2(x_1+x_2)$ columns for $L(1;7,0;6,1)$ and  $L(1;7,0;8,1)$}
    \label{sproadic cases of main theorem 2 with b=1}
\end{figure}

\begin{proof}
Figure \ref{sproadic cases of main theorem 2 with b=1} shows the first $4(7+6)$ and $4(8+7)$ columns for $L(1;7,0;6,1)$ and  $L(1;7,0;8,1)$, respectively. We will discuss the former case ($x_2=x_1-1$), and leave the latter to the reader.

The first $x_1$ columns are all $(0,1)$, and the next $x_2$ columns are all $(2,3)$.  

Consider the $k^{th}$ $(x_1+x_2)$-block after the first one.  We describe how it is generated by the previous block.  In the first $x_1$ columns, the initial $k$ and final $k$ columns are flipped.  In the last $x_2$ columns,  the initial $k-1$ and final $k-1$ columns are flipped, while the $k^{th}$  and $(x_1+x_2-k)^{th}$ columns become $(0,1)$ and $(1,0)$, respectively.

Thus the $\frac{x_1+1}{2}^{th}$ 
$(x_1+x_2)$-block consists of alternating sequences of $0$s and $1$s. 
Hence the eventual horizontal period is $2$.
\end{proof}

\begin{prop}\label{prop - sporadic3}
    The following games have the following horizontal preperiods and periods:

    \begin{center}
\begin{tabular}{c|l|c|l}
   \; &\; &preperiod& period\\
    \hline
    i&$L(2; 2, 0; 3, 1)$& $10$ & $4$\\\hline
    ii&$L(2; 2, 0; 3, 3)$ & $10$ & $4$ \\\hline
    iii&$L(3; 3, 0; 2, 1)$ & $5$ & $6$ \\\hline
    iv&$L(3; 3, 0; 2, 5)$ & $5$ & $6$     \\\hline
\end{tabular}
\end{center}

       
\end{prop}

\begin{figure}[h]
 $\begin{array}{>{\columncolor{yellow!20}}c>{\columncolor{yellow!20}}c>{\columncolor{yellow!20}}c>{\columncolor{yellow!20}}c>{\columncolor{yellow!20}}c>{\columncolor{yellow!20}}c>{\columncolor{yellow!20}}c>{\columncolor{yellow!20}}c>{\columncolor{yellow!20}}c>{\columncolor{yellow!20}}ccccc}
         0&0&1&1&2&0&0&1&2&0&1&1&0&0\\ 0&0&1&2&0&1&1&0&0&2&1&0&0&1\\
         1&1&0&0&3&1&1&0&3&1&0&0&1&1\\ 1&1&0&3&1&0&0&1&1&3&0&1&1&0   
    \end{array}$\qquad
    $\begin{array}{>{\columncolor{yellow!20}}c>{\columncolor{yellow!20}}c>{\columncolor{yellow!20}}c>{\columncolor{yellow!20}}c>{\columncolor{yellow!20}}c>{\columncolor{yellow!20}}c>{\columncolor{yellow!20}}c>{\columncolor{yellow!20}}c>{\columncolor{yellow!20}}c>{\columncolor{yellow!20}}ccccc}
         0&0&1&2&0&1&1&0&0&2&1&0&0&1\\ 0&0&1&1&2&0&0&1&2&0&1&1&0&0\\
         1&1&0&3&1&0&0&1&1&3&0&1&1&0\\ 1&1&0&0&3&1&1&0&3&1&0&0&1&1   
    \end{array}$

\medskip

$\begin{array}{>{\columncolor{yellow!20}}c>{\columncolor{yellow!20}}c>{\columncolor{yellow!20}}c>{\columncolor{yellow!20}}c>{\columncolor{yellow!20}}ccccccc}
         0&0&1&1&2&0&0&1&1&1&0\\ 0&0&1&1&1&0&0&0&1&1&1\\ 0&0&0&2&1&1&0&0&0&1&1\\
         1&1&0&0&3&1&1&0&0&0&1\\ 1&1&0&0&0&1&1&1&0&0&0\\ 1&1&1&3&0&0&1&1&1&0&0   
    \end{array}$\qquad
    $\begin{array}{>{\columncolor{yellow!20}}c>{\columncolor{yellow!20}}c>{\columncolor{yellow!20}}c>{\columncolor{yellow!20}}c>{\columncolor{yellow!20}}ccccccc}
         0&0&0&2&1&1&0&0&0&1&1\\ 0&0&1&1&1&0&0&0&1&1&1\\ 0&0&1&1&2&0&0&1&1&1&0\\
         1&1&1&3&0&0&1&1&1&0&0\\ 1&1&0&0&0&1&1&1&0&0&0\\ 1&1&0&0&3&1&1&0&0&0&1   
    \end{array}$
    \caption{The preperiods and initial periods for $L(2;2,0;3,1)$, $L(2;2,0;3,3)$, $L(3;3,0;2,1)$, and $L(3;3,0;2,5)$}
    \label{sproadic cases of main theorem 2}
\end{figure}

\begin{proof}
    See Figure \ref{sproadic cases of main theorem 2}.  In each case, the preperiod is highlighted, and one period follows.  
\end{proof}

\begin{rem}These sporadic cases arise from a sort of ``eventual'' vector elimination, where, after the preperiod, the \textsc{nim}-values line up so that $(x-x_1,y+y_1)$ and $(x-x_2,y+y_2)$ always have the same values.
These cases were found via computer models, and we currently have no idea how to predict this phenomenon systematically.  In fact, our inability to show when this does \emph{not} happen is our main roadblock to proving Conjecture \ref{conjecture - main2}.
\end{rem}  


\subsection{The main conjecture}

Our computer aided calculations show the following up to $b,x_1,x_2,y_2\leq 20$.  For our Maple code, see the last author's website \cite{Thotsaporn}.

\begin{conj}\label{conjecture - main2}
    Suppose $0<y_2<2b$, $\gcd\{x_1,x_2\}=1$, and  $\gcd\{b,y_2\}=1$.  Suppose further that none of Corollary \ref{cor - vector elimination} or Propositions \ref{prop - Sporadic1}(i),\ref{prop - Sporadic1}(iii), \ref{prop - the trivial case generalized}, \ref{prop - sporadic2}, or \ref{prop - sporadic3} apply.     
    Then the  horizontal period of $L(b;x_1,0;x_2,y_2)$ is given by
    \[g(b,x_1,0,x_2,y_2) =\frac{2b(x_1+x_2)}{\gcd\{2b,y_2\}}.\]
\end{conj}




In the general case of $L(b;x_1,y_1;x_2,y_2)$, we have the exact same phenomenon.
However, it is more complicated to determine when we have eventual vector elimination, so we do not make a more specific conjecture in this paper.

\section{Lengyel's other conjecture: multiple transfer options}\label{section - lengyel's other conjecture}
Lengyel had one other conjecture in \cite{Lengyel}, and we prove it now.  

\begin{defi}
Define the \emph{Lengyel multitransfer game} $L^*(a,b,c)$ to be the vector game defined by 
\[\{(-a,0),(0,-b),(-1,1), (-2,2), \ldots, (-c,c)\}\]     
\end{defi}

Note that we still have \textsc{nim}-periodicity in $(0,-b)$.

\begin{prop}
    For $b\geq 2$, the horizontal component of the period of $L^*(b,b,b)$ is $b+2$.
\end{prop}

Notice that $L^*(1,1,1)=L(1;1,0;1,1)$. 

\begin{proof}
    We give a proof by diagram in Figure \ref{fig: Multioption} for the cases $b=6$ and $b=7$.  Note that in each case, there is no preperiod.  The remaining cases are similar.
\end{proof}

\begin{figure}[h]
    \centering
        $\begin{tblr}{cccccccc}
         0 & 1 & 2 & 3 & 4 & 5 & 6 & 7  \\
         0 & 1 & 2 & 3 & 4 & 5 & 6 & 7  \\
         0 & 1 & 2 & 3 & 4 & 5 & 6 & 7  \\
         0 & 1 & 2 & 3 & 4 & 5 & 6 & 7  \\
         0 & 1 & 2 & 3 & 4 & 5 & 6 & 7  \\
         0 & 0 & 2 & 2 & 4 & 4 & 6 & 6  \\
         \hline[dashed]
         1 & 0 & 3 & 2 & 5 & 4 & 7 & 6 \\
         1 & 0 & 3 & 2 & 5 & 4 & 7 & 6 \\
         1 & 0 & 3 & 2 & 5 & 4 & 7 & 6 \\
         1 & 0 & 3 & 2 & 5 & 4 & 7 & 6 \\
         1 & 0 & 3 & 2 & 5 & 4 & 7 & 6 \\
         1 & 1 & 3 & 3 & 5 & 5 & 7 & 7 
    \end{tblr}$
\qquad
    $\begin{tblr}{ccccccccc}
         0 & 1 & 2 & 3 & 4 & 5 & 6 & 7 & 8 \\
         0 & 1 & 2 & 3 & 4 & 5 & 6 & 7 & 8 \\
         0 & 1 & 2 & 3 & 4 & 5 & 6 & 7 & 8 \\
         0 & 1 & 2 & 3 & 4 & 5 & 6 & 7 & 8 \\
         0 & 1 & 2 & 3 & 4 & 5 & 6 & 7 & 8 \\
         0 & 1 & 2 & 3 & 4 & 5 & 6 & 7 & 8 \\
         0 & 0 & 2 & 2 & 4 & 4 & 6 & 6 & 8 \\
         \hline[dashed]
         1 & 0 & 3 & 2 & 5 & 4 & 7 & 6 & 9\\
         1 & 0 & 3 & 2 & 5 & 4 & 7 & 6 & 9\\
         1 & 0 & 3 & 2 & 5 & 4 & 7 & 6 & 9\\
         1 & 0 & 3 & 2 & 5 & 4 & 7 & 6 & 9\\
         1 & 0 & 3 & 2 & 5 & 4 & 7 & 6 & 9\\
         1 & 0 & 3 & 2 & 5 & 4 & 7 & 6 & 9\\
         1 & 1 & 3 & 3 & 5 & 5 & 7 & 7 & 9
    \end{tblr}$
    \caption{A period of the \textsc{nim}-values for $L^*(6,6,6)$ and $L^*(7,7,7)$}
    \label{fig: Multioption}
\end{figure}

\section{Other Directions.  Open Questions.}\label{section - open questions}

While we are happy that Theorem \ref{Main Theorem 1} is a concrete generalization of \cite{Lengyel} and how Theorem \ref{conjecture - big one} applies to the most general case, we are frustrated at how close the proof to Conjecture \ref{conjecture - main2} feels.  Nevertheless, we feel that we have proved almost all of it, and laid the groundwork for someone else to complete the job.  
\begin{question}
    Exactly how does the `eventual vector elimination' work in Conjecture \ref{conjecture - main2}?
\end{question}
Perhaps they could also answer (and prove?) the following.
\begin{question}
What is the correct conjecture for $L(b;x_1,y_1;x_2,y_2)$ when $y_1>0$?    
\end{question}

Also, section \ref{section - lengyel's other conjecture} provides fertile ground for exploration.  For example:
\begin{question}
What is the horizontal period for $L^*(a,b,c)$ in general?    
\end{question}

If one relaxes the ciliary muscles as if to view an autostereogram, one sees that this paper is about the interplay between different notions of periodicity: horizontal periodicity, \textsc{nim}-periodicity, two-move periodicity, and diagonal periodicity.  We end with three examples which invite the reader to have as much fun with periodicity as we have had.

\subsection*{Adding transfer options}
The vector game $\{(0,-3),(-2,0),(-1,3),(-2,2),(-4,1)\}$ is like a Lengyel transfer game with two extra transfer options.  Thus it also has \textsc{nim}-periodicity and eventual horizontal periodicity.  In fact, the (horizontal) preperiod and period are $14$ and $15$, respectively. Figure \ref{fig: (0,-3),(-2,0),(-1,3),(-2,2),(-4,1)} shows the preperiod and one period, as highlighted in the figure.
\begin{figure}[h]
    \centering
    $\begin{tblr}{cccccccccccccc|ccccccccccccccc}
         0&0&2&1&1&3&3&1&4&0&0&2&1&3&\SetCell{lightgray}0&\SetCell{lightgray}0&\SetCell{lightgray}2&\SetCell{lightgray}1&\SetCell{lightgray}1&\SetCell{yellow}2&\SetCell{yellow}0&\SetCell{yellow}3&\SetCell{yellow}1&\SetCell{yellow}4&\SetCell{pink}0&\SetCell{pink}0&\SetCell{pink}3&\SetCell{pink}1&\SetCell{pink}2\\
         0&0&2&2&1&1&3&0&4&1&1&2&0&0&2&1&3&0&2&\SetCell{cyan}1&\SetCell{cyan}1&\SetCell{cyan}2&\SetCell{cyan}0&\SetCell{cyan}3&\SetCell{green}1&\SetCell{green}1&\SetCell{green}2&\SetCell{green}0&\SetCell{green}0\\
         0&0&2&2&0&0&2&1&1&3&0&2&1&4&\SetCell{pink}0&\SetCell{pink}0&\SetCell{pink}2&\SetCell{pink}1&\SetCell{pink}3&\SetCell{lightgray}0&\SetCell{lightgray}0&\SetCell{lightgray}2&\SetCell{lightgray}1&\SetCell{lightgray}1&\SetCell{yellow}2&\SetCell{yellow}0&\SetCell{yellow}3&\SetCell{yellow}1&\SetCell{yellow}2\\
         1&1&3&0&0&2&2&0&5&1&1&3&0&2&\SetCell{green}1&\SetCell{green}1&\SetCell{green}3&\SetCell{green}0&\SetCell{green}0&3&1&2&0&5&\SetCell{cyan}1&\SetCell{cyan}1&\SetCell{cyan}2&\SetCell{cyan}0&\SetCell{cyan}3\\
         1&1&3&3&0&0&2&1&5&0&0&3&1&1&\SetCell{yellow}3&\SetCell{yellow}0&\SetCell{yellow}2&\SetCell{yellow}1&\SetCell{yellow}3&\SetCell{pink}0&\SetCell{pink}0&\SetCell{pink}3&\SetCell{pink}1&\SetCell{pink}2&\SetCell{lightgray}0&\SetCell{lightgray}0&\SetCell{lightgray}3&\SetCell{lightgray}1&\SetCell{lightgray}1\\
         1&1&3&3&1&1&3&0&0&2&1&3&0&5&\SetCell{cyan}1&\SetCell{cyan}1&\SetCell{cyan}3&\SetCell{cyan}0&\SetCell{cyan}2&\SetCell{green}1&\SetCell{green}1&\SetCell{green}3&\SetCell{green}0&\SetCell{green}0&3&1&2&0&3
    \end{tblr}$
    \caption{The preperiod and one period of the \textsc{nim}-values of the vector game $\{(0,-3),(-2,0),(-1,3),(-2,2),(-4,1)\}$}
    \label{fig: (0,-3),(-2,0),(-1,3),(-2,2),(-4,1)}
\end{figure}
If we squint, the values in the period seem to be following a sort of diagonal periodicity in the direction $(5,-4)$.  Perhaps we redefine diagonal periodicity using $\mathcal{SG}^*(x,y)=2$ whenever $\mathcal{SG}(x,y)\geq 2$?  But then how do we predict the position of the $4$ and $5$ in the tenth column of the period?  And why is $(5,-4)$ the direction to squint in?
\begin{question}
    How many of the present methods apply when we add extra transfer options to Lengyel transfer games?
\end{question}

\subsection*{Mimicing two-move periodicity} We wouldn't expect the vector game $\{(0,-2),(-1,0),(-3,-2),(-2,2)\}$ to have \textsc{nim}-periodicity, and in fact, it does not.  Nor would we expect it to have two-move periodicity, but it sort of does!  Observe Figure \ref{fig: (0,-2),(-1,0),(-3,-2),(-2,2)}.  We see periodicity in the direction $(5,4)$
\begin{figure}[h]
    \centering
    $\begin{tblr}{cccccccccc}
         \SetCell{yellow}0&\SetCell{pink}1&\SetCell{yellow}0&\SetCell{pink}1&\SetCell{yellow}0&\SetCell{pink}1&\SetCell{yellow}0&\SetCell{pink}1&\SetCell{yellow}0&\SetCell{pink}1\\
         \SetCell{yellow}0&\SetCell{pink}1&\SetCell{yellow}0&\SetCell{pink}1&\SetCell{yellow}0&\SetCell{pink}1&\SetCell{yellow}0&\SetCell{pink}1&\SetCell{yellow}0&\SetCell{pink}1\\
         \SetCell{pink}1&\SetCell{yellow}0&\SetCell{pink}1&2&3&2&3&2&3&2\\
         \SetCell{pink}1&\SetCell{yellow}0&\SetCell{pink}1&2&3&2&3&2&3&2\\
         \SetCell{yellow}0&\SetCell{pink}1&\SetCell{yellow}0&3&2&\SetCell{yellow}0&\SetCell{pink}1&\SetCell{yellow}0&\SetCell{pink}1&\SetCell{yellow}0\\
         \SetCell{yellow}0&\SetCell{pink}1&\SetCell{yellow}0&3&2&\SetCell{yellow}0&\SetCell{pink}1&\SetCell{yellow}0&\SetCell{pink}1&\SetCell{yellow}0\\
         \SetCell{pink}1&\SetCell{yellow}0&\SetCell{pink}1&2&3&\SetCell{pink}1&\SetCell{yellow}0&\SetCell{pink}1&2&3\\
         \SetCell{pink}1&\SetCell{yellow}0&\SetCell{pink}1&2&3&\SetCell{pink}1&\SetCell{yellow}0&\SetCell{pink}1&2&3\\
         \SetCell{yellow}0&\SetCell{pink}1&\SetCell{yellow}0&3&2&\SetCell{yellow}0&\SetCell{pink}1&\SetCell{yellow}0&3&2\\
         \SetCell{yellow}0&\SetCell{pink}1&\SetCell{yellow}0&3&2&\SetCell{yellow}0&\SetCell{pink}1&\SetCell{yellow}0&3&2
    \end{tblr}$
    \caption{Some of the \textsc{nim}-values of the vector game $\{(0,-2),(-1,0),(-3,-2),(-2,2)\}$}
    \label{fig: (0,-2),(-1,0),(-3,-2),(-2,2)}
\end{figure}

\begin{question}
    Why is $(5,4)$ the direction of periodicity of the vector game $\{(0,-2),(-1,0),(-3,-2),(-2,2)\}$?  What is the class of vector games that mimic two-move periodicity in this way?
\end{question}

\subsection*{The most general case}
The vector game $\{(0,-3),(-2,0),(-1,3),(-2,2),(-3,1),(-1,-2)\}$ does not exhibit any kind of periodicity that we can pin down, though there are some hints at regularity.   For example, the first $9$ columns exhibit \textsc{nim}-periodicity. See Figure \ref{fig: (0,-3),(-2,0),(-1,3),(-2,2),(-3,1),(-1,-2)}, where the first failure of \textsc{nim}-periodicity is highlighted.

\begin{figure}[h]
    \centering
    $\begin{tblr}{cccccccccccccccccccccccccccccc}
    0& 0& 2& 1& 1& 3& 3& 1& 1& \SetCell{yellow}2& 0& 0& 2& 2& 1& 1& 3& 0& 1& 2& 0& 0& 1& 2& 0& 1& 2& 0& 1& 1\\
0& 0& 2& 2& 0& 3& 1& 1& 3& 3& 4& 0& 0& 2& 2& 0& 1& 1& 3& 0& 4& 2& 0& 1& 2& 0& 0& 1& 2& 0\\
0& 2& 2& 0& 0& 2& 2& 1& 3& 3& 1& 1& 2& 0& 0& 2& 2& 1& 1& 3& 0& 1& 1& 2& 0& 1& 2& 0& 1& 2 \\
1& 1& 3& 0& 0& 2& 2& 0& 0& \SetCell{yellow}4& 1& 1& 3& 3& 0& 0& 2& 2& 0& 4& 1& 3& 0& 1& 2& 0& 1& 2& 0& 0\\
1& 1& 3& 3& 1& 2& 0& 0& 2& 2& 5& 5& 3& 1& 1& 3& 0& 0& 2& 2& 5& 1& 2& 0& 1& 1& 2& 0& 1& 2\\
1& 3& 3& 1& 1& 3& 3& 0& 2& 2& 0& 0& 5& 2& 1& 1& 3& 0& 0& 1& 2& 0& 0& 1& 2& 0& 1& 2& 0& 1\\
0& 0& 2& 1& 1& 3& 3& 1& 1& 3& 0& 0& 2& 2& 4& 4& 5& 1& 1& 3& 0& 4& 2& 0& 1& 2& 0& 1& 1& 2\\
0& 0& 2& 2& 0& 3& 1& 1& 3& 3& 4& 1& 2& 0& 0& 2& 2& 4& 1& 1& 3& 0& 1& 2& 0& 0& 1& 2& 0& 1\\
0& 2& 2& 0& 0& 2& 2& 1& 3& 3& 1& 1& 3& 3& 0& 0& 2& 2& 4& 0& 1& 1& 3& 0& 4& 2& 0& 1& 2& 0\\
1& 1& 3& 0& 0& 2& 2& 0& 0& 2& 1& 1& 3& 3& 1& 1& 3& 0& 0& 2& 2& 1& 1& 3& 0& 1& 2& 0& 0& 1\\
1& 1& 3& 3& 1& 2& 0& 0& 2& 2& 5& 0& 3& 1& 1& 3& 3& 5& 0& 0& 2& 2& 0& 1& 1& 3& 0& 4& 2& 0\\
1& 3& 3& 1& 1& 3& 3& 0& 2& 2& 0& 0& 2& 2& 1& 3& 3& 1& 1& 3& 0& 0& 2& 2& 1& 1& 3& 0& 1& 1
    \end{tblr}$
    \caption{Some of the \textsc{nim}-values of the vector game $\{(0,-3),(-2,0),(-1,3),(-2,2),(-3,1),(-1,-2)\}$}
    \label{fig: (0,-3),(-2,0),(-1,3),(-2,2),(-3,1),(-1,-2)}
\end{figure}
\begin{question}
    Is there any order in the chaos of general subtraction-transfer games?
\end{question}

\end{document}